\documentclass[a4paper,11pt,reqno]{amsart}
\usepackage[a4paper,hmargin={2.5cm,2.5cm},vmargin={2cm,2cm}]{geometry}

\usepackage[square]{natbib} 
\usepackage{amssymb}
\usepackage[mathscr]{euscript}
\usepackage{bbm}
\usepackage{dsfont}
\usepackage{stmaryrd}
\usepackage[all,arc]{xy}
\DeclareMathAlphabet{\mathpzc}{OT1}{pzc}{m}{it}

\theoremstyle{plain}
\newtheorem{theorem}{Theorem}[section]
\newtheorem{lemma}[theorem]{Lemma}
\newtheorem{proposition}[theorem]{Proposition}
\newtheorem{corollary}[theorem]{Corollary}

\theoremstyle{definition}
\newtheorem{definition}[theorem]{Definition}
\newtheorem{examples}[theorem]{Examples}
\newtheorem{example}[theorem]{Example}

\theoremstyle{remark}
\newtheorem{remark}[theorem]{Remark}

\newenvironment{eqcond}{\begin{enumerate}\renewcommand{\theenumi}{\roman{enumi}}}{\end{enumerate}}
\renewcommand{\theenumi}{\arabic{enumi}}

\newcommand{\Rw}{\Rightarrow}

\newcommand{\hrw}{\hookrightarrow}

\newcommand{\RLw}{\Leftrightarrow}

\newcommand{\fa}{\mathfrak{a}}

\newcommand{\ff}{\mathfrak{f}}

\newcommand{\fj}{\mathfrak{j}}
\newcommand{\fp}{\mathfrak{p}}
\newcommand{\fq}{\mathfrak{q}}

\newcommand{\fv}{\mathfrak{v}}
\newcommand{\fw}{\mathfrak{w}}
\newcommand{\fx}{\mathfrak{x}}
\newcommand{\fy}{\mathfrak{y}}
\newcommand{\fz}{\mathfrak{z}}

\newcommand{\calA}{\mathcal{A}}
\newcommand{\calB}{\mathcal{B}}

\newcommand{\calO}{\mathcal{O}}

\newcommand{\fX}{\mathfrak{X}}

\DeclareMathOperator{\yoneda}{\mathpzc{y}}
\DeclareMathOperator{\coyoneda}{\mathpzc{h}}
\DeclareMathOperator{\yonedaT}{\mathpzc{Y}}
\DeclareMathOperator{\yonmult}{\mathpzc{m}}
\DeclareMathOperator{\coyonmult}{\mathpzc{w}}
\DeclareMathOperator{\colim}{colim}

\DeclareMathOperator{\Sup}{Sup}
\DeclareMathOperator{\Inf}{Inf}
\DeclareMathOperator{\upc}{\uparrow\!}
\DeclareMathOperator{\downc}{\downarrow\!}
\DeclareMathOperator{\Coalg}{\mathsf{Coalg}}
\DeclareMathOperator{\ev}{ev}

\DeclareMathOperator{\Spl}{Spl}

\DeclareMathOperator{\kar}{kar}

\newcommand{\mate}[1]{\,^\ulcorner\! #1^\urcorner}
\newcommand{\umate}[1]{\,_\llcorner\! #1_\lrcorner}

\newcommand{\fspstrP}[2]{\llbracket #1,#2\rrbracket}
\newcommand{\fspstrV}[2]{\langle\!\langle #1,#2\rangle\!\rangle}

\newcommand{\catfont}[1]{\mathsf{#1}}

\newcommand{\SET}{\catfont{Set}}
\newcommand{\REL}{\catfont{Rel}}
\newcommand{\ORD}{\catfont{Ord}}
\newcommand{\MET}{\catfont{Met}}
\newcommand{\UMET}{\catfont{UMet}}
\newcommand{\TOP}{\catfont{Top}}

\newcommand{\STCOMP}{\catfont{StablyComp}}

\newcommand{\COMPHAUS}{\catfont{CompHaus}}
\newcommand{\SPECDIST}{\catfont{SpecDist}}
\newcommand{\SPEC}{\catfont{Spec}}
\newcommand{\STONEDIST}{\catfont{StoneDist}}
\newcommand{\STONE}{\catfont{Stone}}

\newcommand{\AP}{\catfont{App}}
\newcommand{\UAP}{\catfont{UApp}}

\newcommand{\SUP}{\catfont{Sup}}
\newcommand{\INF}{\catfont{Inf}}
\newcommand{\CONTLAT}{\catfont{ContLat}}
\newcommand{\DLAT}{\catfont{DLat}}
\newcommand{\TAL}{\catfont{Tal}}
\newcommand{\CCD}{\catfont{CCD}}
\newcommand{\BOOL}{\catfont{Bool}}
\newcommand{\HEYT}{\catfont{Heyt}}

\newcommand{\two}{\catfont{2}}
\newcommand{\V}{\mathcal{V}}
\newcommand{\Pp}{\catfont{P}_{_{\!\! +}}}
\newcommand{\Pm}{\catfont{P}_{_{\!\! \wedge}}}

\newcommand{\Mat}[1]{#1\text{-}\catfont{Rel}}
\newcommand{\Mod}[1]{#1\text{-}\catfont{Dist}}

\newcommand{\Cat}[1]{#1\text{-}\catfont{Cat}}
\newcommand{\Gph}[1]{#1\text{-}\catfont{Gph}}
\newcommand{\CatSep}[1]{#1\text{-}\catfont{Cat}_\mathrm{sep}}
\newcommand{\Cocts}[1]{#1\text{-}\catfont{CoCts}}
\newcommand{\Cts}[1]{#1\text{-}\catfont{Cts}}

\newcommand{\Repr}[1]{#1\text{-}\catfont{ReprCat}}
\newcommand{\ReprDist}[1]{#1\text{-}\catfont{ReprDist}}

\newcommand{\relto}{{\longrightarrow\hspace*{-2.8ex}{\mapstochar}\hspace*{2.6ex}}}
\newcommand{\modto}{{\longrightarrow\hspace*{-2.8ex}{\circ}\hspace*{1.2ex}}}
\newcommand{\kto}{\relbar\joinrel\rightharpoonup}
\newcommand{\krelto}{\,{\kto\hspace*{-2.5ex}{\mapstochar}\hspace*{2.6ex}}}
\newcommand{\kmodto}{\,{\kto\hspace*{-2.8ex}{\circ}\hspace*{1.3ex}}}
\newcommand{\xrelto}[1]{{\xrightarrow{#1}\hspace*{-2.8ex}{\mapstochar}\hspace*{2.8ex}}}

\newcommand{\blackright}{\mbox{ $-\!\mbox{\footnotesize $\bullet$}$ }}
\newcommand{\blackleft}{\mbox{ $\mbox{\footnotesize $\bullet$}\!-$ }}
\newcommand{\whiteright}{\multimap}
\newcommand{\whiteleft}{\mbox{ $\mbox{\footnotesize $\circ$}\!-$ }}

\newcommand{\monadfont}[1]{\mathbbm{#1}}

\newcommand{\mT}{\monadfont{T}}

\newcommand{\mU}{\monadfont{U}}
\newcommand{\mV}{\monadfont{V}}
\newcommand{\mtV}{\widetilde{\monadfont{V}}}
\newcommand{\mL}{\monadfont{L}}
\newcommand{\mP}{\monadfont{P}}

\newcommand{\monad}{(T,e,m)}
\newcommand{\umonad}{(U,e,m)}
\newcommand{\wmonad}{(L,e,m)}
\newcommand{\vmonad}{(V,\coyoneda,\coyonmult)}
\newcommand{\pmonad}{(P,\yoneda,\yonmult)}
\newcommand{\tvmonad}{(\widetilde{V},\widetilde{\coyonmult},\widetilde{\coyoneda})}

\newcommand{\theoryfont}[1]{\mathscr{#1}}
\newcommand{\Tth}{\theoryfont{T}}
\newcommand{\Uth}{\theoryfont{U}}
\newcommand{\Ith}{\theoryfont{I}}
\newcommand{\toptheory}{(\mT,\V,\xi)}

\newcommand{\Txi}{T_\xi}
\newcommand{\Uxi}{U_\xi}

\newcommand\adjunct[2]{\xymatrix@=8ex{\ar@{}[r]|{\top}\ar@<1mm>@/^2mm/[r]^{{#2}} & \ar@<1mm>@/^2mm/[l]^{{#1}}}}

\newcommand{\eps}{\varepsilon}
\newcommand{\op}{\mathrm{op}}

\newcommand{\sep}{\mathrm{sep}}
\newcommand{\spl}{\mathrm{spl}}
\newcommand{\kleisli}{\circ}

\newcommand{\trunminus}{\ominus}
\newcommand{\Tast}{\circledast}
\newcommand{\homV}{\hom_\xi}

\newcommand{\field}[1]{\mathds{#1}}

\newcommand{\N}{\field{N}}

\title{The enriched Vietoris monad on representable spaces}
\author{Dirk Hofmann}
\thanks{Partial financial assistance by FEDER funds through COMPETE -- Operational Programme Factors of Competitiveness (Programa Operacional Factores de Competitividade) and by Portuguese funds through the Center for Research and Development in Mathematics and Applications (University of Aveiro) and the Portuguese Foundation for Science and Technology (FCT -- Funda\c{c}\~ao para a Ci\^encia e a Tecnologia), within project PEst-C/MAT/UI4106/2011 with COMPETE number FCOMP-01-0124-FEDER-022690, and the project MONDRIAN under the contract PTDC/EIA-CCO/108302/2008 with COMPETE number FCOMP-01-0124-FEDER-010047 is gratefully acknowledge.}
\address{CIDMA -- Center for Research and Development in Mathematics and Applications, Department of Mathematics, University of Aveiro, 3810-193 Aveiro, Portugal}
\email{dirk@ua.pt}
\date{\today}

\usepackage[pdftex]{hyperref}

\begin{document}

\begin{abstract}
Employing a formal analogy between ordered sets and topological spaces, over the past years we have investigated a notion of cocompleteness for topological, approach and other kind of spaces. In this new context, the down-set monad becomes the filter monad, cocomplete ordered set translates to continuous lattice, distributivity means disconnectedness, and so on. Curiously, the dual(?) notion of completeness does not behave as the mirror image of the one of cocompleteness; and in this paper we have a closer look at complete spaces. In particular, we construct the ``up-set monad'' on representable spaces (in the sense of L.\ Nachbin for topological spaces, respectively C.\ Hermida for multicategories); we show that this monad is of Kock-Z\"oberlein  type; we introduce and study a notion of weighted limit similar to the classical notion for enriched categories; and we describe the Kleisli category of our ``up-set monad''. We emphasize that these generic categorical notions and results can be indeed connected to more ``classical'' topology: for topological spaces, the ``up-set monad'' becomes the lower Vietoris monad, and the statement ``$X$ is totally cocomplete if and only if $X^\op$ is totally complete'' specialises to O.\ Wyler's characterisation of the algebras of the Vietoris monad on compact Hausdorff spaces.
\end{abstract}

\maketitle

\section*{Introduction}

In this paper we continue the work presented in \citep{Hof11} on ``injective spaces via adjunction'' whose fundamental aspect can be described by the slogan \emph{topological spaces are categories}, and therefore can be studied using notions and techniques from (enriched) Category Theory. The use of the word ``continue'' here is slightly misleading as we do not follow directly the path of \citep{Hof11} but rather develop ``the second aspect'' of this theory. To explain this better, recall that an order relation on a set $X$ defines a monotone map of type
\[
 X^\op\times X\to\two,
\]
and from that one obtains the two Yoneda embeddings
\begin{align*}
X\to\two^{X^\op}=:PX &&\text{and}&& X\to(\two^X)^\op=:VX.
\end{align*}
Furthermore, both $PX$ and $VX$ are part of monads $\mP$ and $\mV$ on $\ORD$ (the category of ordered sets and monotone maps) with Eilenberg--Moore categories $\ORD^{\mP}\simeq\SUP$ (the category of complete lattices and $\sup$-preserving maps) and $\ORD^{\mV}\simeq\INF$ (the category of complete lattices and $\inf$-preserving maps) respectively. One has full embeddings
\begin{align*}
 \ORD_{\mP}\to\ORD^{\mP} &&\text{and}&& \ORD_{\mV}\to\ORD^{\mV}
\end{align*}
from the Kleisli categories into the Eilenberg--Moore categories, and from that one obtains an equivalence (see \citep{RW94,RW04})
\begin{align*}
 \kar(\ORD_{\mP})\simeq \CCD_{\sup} &&\text{and}&&
 \kar(\ORD_{\mV})\simeq \CCD_{\inf}
\end{align*}
between the idempotent split completion of the Kleisli categories on one side and the categories of completely distributive complete lattice and $\sup$-preserving, respectively $\inf$-preserving, maps on the other. These equivalences restrict to
\begin{align*}
 \ORD_{\mP}\simeq \TAL_{\sup} &&\text{and}&&
 \ORD_{\mV}\simeq \TAL_{\inf},
\end{align*}
where ``Tal'' stands for totally algebraic lattice. Finally, having both sides restricted to adjoint morphisms leads to the equivalence
\[
  \ORD^\op\simeq\TAL
\]
between the dual category of $\ORD$ and the category $\TAL$ of totally algebraic lattices and $\sup$- and $\inf$-preserving maps.

In \citep{Hof11,Hof13} and \citep{CH09a} we followed the path on the left described above, but now with geometric objects like topological or approach spaces in lieu of ordered sets (the latter representing ``metric'' topological spaces, see \citep{Low97}). To illustrate this analogy, note that the ultrafilter convergence of a topological space defines a continuous map
\[
 (UX)^\op\times X\to \two
\]
(where $\mathsf{2}$ is the Sierpi\'nski space and $(UX)^\op$ is explained in Section \ref{sect:ReprDual}). Moreover, the space $(UX)^\op$ turns out to be exponentiable, therefore we obtain the Yoneda embedding
\[
 \yoneda_X:X\to\mathsf{2}^{(UX)^\mathrm{op}}=:PX.
\]
The ``story of cocompleteness'' can now be told almost as for ordered sets, and we refer to the above-mentioned papers for detailed information. However, in contrast to the ordered case, the subsequent development of the right side cannot be seen as the dual image of the left side; and it is the aim of this work to explore this path.

The paper is organised as follows.
\begin{itemize}
\item In Section \ref{sect:setting} we describe our general framework, namely that of a topological theory $\Tth=\toptheory$ (see \citep{Hof07}) consisting of a monad $\mT=\monad$ on $\SET$, a quantale $\V$ and a $\mT$-algebra structure $\xi:T\V\to\V$ on $\V$. The associated notion of $\Tth$-category embodies several types of spaces such as topological, metric or approach spaces, and together with $\Tth$-functors and $\Tth$-distributors defines the categories $\Cat{\Tth}$ and $\Mod{\Tth}$ respectively. We recall succinctly the main constructions and results, in particular that core-compactness implies tensor-exponentiability. In this context, for topological spaces we give a variation of Alexander's Subbase Lemma for core-compactness using a simple convergence-theoretic argument (Example \ref{ex:AlexSubbase}).
\item Section \ref{sect:ReprDual} is devoted to the important notion of representable $\Tth$-categories (Definition \ref{def:Repr}) defined as precisely the pseudo-algebras for a natural lifting of the $\SET$-monad $\mT$ to a monad of Kock-Z\"oberlein type on $\Cat{\Tth}$. We also introduce the concept of a dualisable $\Tth$-graph (Definition \ref{def:dualTGph}). Our interest in representable $\Tth$-categories derives from the fact that these are precisely those $\Tth$-categories for which the associated dual $\Tth$-graph is a $\Tth$-category (Definition \ref{def:dualOfTGph} and Proposition \ref{prop:DualVsCore}). For topological T$_0$-spaces, the concept of representability specialises to the classical notion of a stably compact space which is closely related to L.\ Nachbin's ordered compact Hausdorff spaces.
\item In Section \ref{sect:CoCts} we recall the principal facts about weighted colimits and cocomplete $\Tth$-categories obtained in \citep{Hof11} and \citep{CH09a}. We stress that, unlike the classical case of enriched categories, here it is necessary to consider weights with arbitrary codomain, not just the one-element category $G$. Compared to previous work we change notation and use the designation ``totally cocomplete'' for a $\Tth$-category admitting all weigthed colimits, and say that a $\Tth$-category is ``cocomplete'' whenever it has all those weighted colimits where the codomain of the weight is $G$.
\item From Section \ref{sect:Vietoris} on we assume that our monad $\mT=\monad$ satisfies $T1=1$. In this section we show that the exponential $\V^X$ is always a dualisable $\mT$-graph and, in the topological case, its dual $(\two^X)^\op$ turns out to be the lower Vietoris topological space. We point out how this can be used to deduce the classical characterisation of exponentiable topological spaces as precisely the core-compact ones (Example \ref{ex:UpperVietoris}). The main results of this section state that the construction $X\mapsto(\V^X)^\op=:VX$ leads to a monad $\mV=\vmonad$ of Kock-Z\"oberlein type on both $\Cat{\Tth}$ and $\Repr{\Tth}$ (the category of representable $\Tth$-categories and pseudo-homomorphisms), see Theorem \ref{thm:Vietoris}. 
\item In Section \ref{sect:Cts} we analyse the notion of weighted limit in $\Tth$-categories.
\item Section \ref{sect:Isbell} lifts the classical adjunction between ``up-sets'' and ``down-sets'' (see \citep[Section 5]{Woo04}) into the realm of $\Tth$-categories.
\item In Section \ref{sect:TotCts} we introduce totally complete $\Tth$-categories and show that they are precisely the duals of totally cocomplete $\Tth$-categories (Theorem \ref{thm:TotCoComplDual} and Examples \ref{ex:TotCoComplDual}).
\item Finally, in Section \ref{sect:Kleisli} we give a characterisation of the morphisms of the Kleisli category $\Repr{\Tth}_\mV$ of $\mV$. We also observe how the notion of an Esakia space arises naturally in this context via splitting idempotents of the full subcategory of $\Repr{\Tth}_\mV$ defined by all $\mT$-algebras. We find it worthwhile to mention that this implies in particular that the category $\HEYT_{\bot,\vee}$ of Heyting algebras and finite suprema preserving maps is the idempotent split completion of the category $\BOOL_{\bot,\vee}$ of Boolean algebras and finite suprema preserving maps (Example \ref{ex:HeytSplCpl}). 
\end{itemize}

\section{The setting}\label{sect:setting}

In this paper we will work with $\Tth$-categories, $\Tth$-functors and $\Tth$-distributors, for a (strict) topological theory $\Tth$. Below we recall some of the main facts and refer to \citep{Hof07}, \citep{Hof11} and \citep{CH09a} for details.

\begin{definition}
A \emph{topological theory} $\Tth=\toptheory$ consists of:
\begin{enumerate}
\item a monad $\mT=\monad$ on $\SET$ (with multiplication $m$ and unit $e$),
\item a commutative and unital quantale $\V=(\V,\otimes,k)$,
\item a function $\xi:T\V\to\V$,
\end{enumerate}
such that
\begin{enumerate}\renewcommand{\theenumi}{\alph{enumi}}
\item $T$ preserves weak pullbacks and each naturality square of $m$ is a weak pullback,
\item\label{condition} the pair $(\V,\xi)$ is an Eilenberg--Moore algebra for $\mT$ and the monoid structure on $\V$ in $(\SET,\times,1)$ lifts to a monoid structure on $(\V,\xi)$ in $(\SET^\mT,\times,1)$, that is, the following diagrams have to commute:
\[
\xymatrix@C=10ex{
T(\V\times \V)\ar[r]^{T(-\otimes-)}\ar[d]_{\langle\xi\cdot T\pi_1,\xi\cdot T\pi_2\rangle} & T\V\ar[d]^{\xi} & T1\ar[d]^{!}\ar[l]_{T(k)} \\
\V\times\V\ar[r]_{-\otimes-} & \V & 1,\ar[l]^{k}}
\]
\item writing $P_{\V}:\SET\to\ORD$ for the functor that sends a function $f:X\to Y$ to the left adjoint of the ``inverse image'' function $f^{-1}:\V^Y\to\V^X,\;\varphi\mapsto\varphi\cdot f$ (where $\V^X$ is the set of functions from $X$ to $\V$, with pointwise order), the functions $\xi_X:V^X\to V^{TX},\;f\mapsto \xi\cdot Tf$ (for $X$ in $\SET$) are the components of a natural transformation $(\xi_X)_X:P_{\V}\to P_{\V}T$.
\end{enumerate}
\end{definition}
\begin{remark}\label{rem:homxi}
As shown in \citep[Lemma 3.2]{Hof07}, the \emph{internal hom} in $\V$ defined by
\[x\otimes y\leq z\iff x\leq\hom(y,z)\]
automatically satisfies
\[
\xymatrix{T(\V\times\V)\ar[rr]^-{T(\hom)}\ar[d]_{\langle\xi\cdot T\pi_1,\xi\cdot T\pi_2\rangle}\ar@{}[drr]|{\ge} && T\V\ar[d]^\xi\\ \V\times\V\ar[rr]_-{\hom} && \V.}
\]
\end{remark}
\begin{remark}
Our notation differs here from \citep{Hof07} where $\xi$ is only assumed to be a lax Eilenberg--Moore structure on $\V$ and the diagrams in \eqref{condition} are only required to commute laxly. A theory satisfying the stronger conditions above is called \emph{strict topological theory} there. However, in this paper all theories are assumed to be strict, therefore we simply use the term ``topological theory''.
\end{remark}
Throughout this paper we will assume that a topological theory $\Tth=\toptheory$ is given. Moreover, we will \emph{always assume that $\V$ is non-trivial}, that is, $\bot\neq k$. Consequently, since $\V$ is assumed to be a $\mT$-algebra, the monad $\mT$ must be non-trivial. We recall here that there are two trivial monads $\mT=\monad$ on $\SET$: one with $TX=1$ for every set $X$, and one with $TX=1$ for every non-empty set and $T\varnothing=\varnothing$. A monad $\mT=\monad$ different from these two is called non-trivial. For a non-trivial monad, $T$ is faithful and $e$ is point-wise injective (see \citep{Man76}, for instance).

Our leading examples are the following:
\begin{examples}\label{ex:Theories}
\mbox{}
\begin{enumerate}
\item For any quantale $\V$ we can consider the theory whose monad-part is the identity monad on $\SET$ and where $\xi:\V\to\V$ is the identity function. We write this trivial topological theory as $\Ith_{\V}$ or simply as $\Ith$.
\item Let $\V$ be the $2$-element chain $\two$, and consider the ultrafilter monad $\mU=\umonad$ on $\SET$. This together with the ``identity'' function $\xi:U\two\to\two$ is a topological theory which we denote by $\Uth_{\two}$.
\item More general, for a non-trivial monad $\mT=\monad$ on $\SET$ where $T$ preserves weak pullbacks and each naturality square of $m$ is a weak pullback and every completely distributive complete lattice $\V$ (considering $\otimes=\wedge$ and $k=\top$), $\toptheory$ is a topological theory where
\[
\xi:T\V\to\V,\;\;\fx\mapsto\bigvee\{v\in\V\mid\fx\in T(\upc v)\}.
\]
\item\label{ex:UMetTh} In particular, for the ultrafilter monad $\mU=\umonad$ on $\SET$ and the complete lattice $[0,\infty]$ ordered by the ``greater or equal'' relation $\ge$ (so that the infimum of two numbers is there maximum and the supremum of $S\subseteq[0,\infty]$ is given by $\inf S$), we write $\Pm=([0,\infty],\max,0)$ for the corresponding quantale and $\Uth_{\Pm}=(\mU,\Pm,\xi)$ for the corresponding theory where
\[
\xi:U([0,\infty])\to[0,\infty],\;\;\fx\mapsto\inf\{v\in[0,\infty]\mid[0,v]\in\fx\}.
\]
Also note that
\[
 \hom(u,v)=
\begin{cases}
 0 & \text{if $u\ge v$,}\\
 v & \text{otherwise}
\end{cases}
\]
in $\Pm$.
\item\label{ex:MetTh} Let $\V$ be the quantale $\Pp=([0,\infty],+,0)$ of extended non-negative real numbers ordered by the ``greater or equal'' relation (see \citep{Law73}), and consider again the ultrafilter monad $\mU=\umonad$ on $\SET$. Together with the function $\xi:U([0,\infty])\to[0,\infty]$ as above this makes up a topological theory, denoted as $\Uth_{\Pp}$. For later use we record here that the internal hom of the quantale $\Pp$ is given by truncated minus:
\[\hom(u,v)=v\ominus u:=\max\{v-u,0\}.\]
\item\label{WordTh} For any quantale $\V$, the word monad $\mL=\wmonad$ on $\SET$ together with the function
\[\xi:L(\V)\to\V,\;(v_1,\ldots,v_n)\mapsto v_1\otimes\ldots\otimes v_n, \; (\ )\mapsto k\]
determine a topological theory.
\end{enumerate}
\end{examples}

Since some of our principal examples involve the ultrafilter monad, we note here two important results.

\begin{proposition}\label{prop:UltExtExc}
Let $X$ be a set, $\ff$ a filter and $\fj$ an ideal on $X$ with $\ff\cap\fj=\varnothing$. Then there exists an ultrafilter $\fx$ on $X$ with $\ff\subseteq\fx$ and $\fx\cap\fj=\varnothing$.
\end{proposition}
\begin{proof}
See \citep[Theorem 6]{Sto38}, for instance.
\end{proof}

\begin{theorem}[\citep{Man69}]
The Eilenberg--Moore category $\SET^\mU$ of the ultrafilter monad on $\SET$ is equivalent to the category $\COMPHAUS$ of compact Hausdorff spaces and continuous maps.
\end{theorem}

Every topological theory allows for a number of constructions and definitions which were succinctly recalled in \citep{CH09a}. Below we give a slightly revised version of \citep[Section 1]{CH09a}.

\subsection*{I}
\newcounter{subsect:one}
\setcounter{subsect:one}{1}
The quantaloid $\Mat{\V}$ (see \citep{BCSW83}) has sets as objects, and a morphism $r:X\relto Y$ from $X$ to $Y$ is a \emph{$\V$-relation} $r:X\times Y\to\V$ (also called $\V$-matrix). The composition of $\V$-relations $r:X\relto Y$ and $s:Y\relto Z$ is defined as matrix multiplication
\[s\cdot r(x,z)=\bigvee_{y\in Y}r(x,y)\otimes s(y,z),\]
and the identity arrow $1_X:X\relto X$ is the $\V$-relation which sends all diagonal elements $(x,x)$ to $k$ and all other elements to the bottom element $\bot$ of $\V$. The set $\Mat{\V}(X,Y)$ of all $\V$-relations from $X$ to $Y$ becomes a complete ordered set by putting
\[
r\le r'\;\text{whenever}\; \forall x\in X\;\forall y\in Y\;.\;r(x,y)\le r'(x,y),
\]
for $\V$-relations $r,r':X\relto Y$; composition from either side preserves this order.

The category $\Mat{\V}$ has an involution $(r:X\relto Y)\mapsto (r^\circ:Y\relto X)$ where $r^\circ(y,x)=r(x,y)$, satisfying
\begin{align*}
1_X^\circ=1_X, && (s\cdot r)^\circ=r^\circ\cdot s^\circ, && {r^\circ}^\circ=r,
\end{align*}
as well as $r^\circ\le s^\circ$ whenever $r\le s$. Furthermore, there is a faithful functor
\[\SET\to\Mat{\V},\;(f:X\to Y)\mapsto (f:X\relto Y)\]
sending a map $f:X\to Y$ to its graph $f:X\relto Y$ defined by
\[
f(x,y)=
\begin{cases}
k & \text{if $f(x)=y$,}\\
\bot & \text{else.}
\end{cases}
\]
In the sequel we will not distinguish between the function $f$ and the $\V$-relation $f$ and simply write $f:X\to Y$. We also note that $f\dashv f^\circ$ in the quantaloid $\Mat{\V}$.

Let $t:X\relto Z$ be a $\V$-relation. The composition functions
\begin{align*}
-\cdot t:\Mat{\V}(Z,Y)\to\Mat{\V}(X,Y) &&\text{and}&&
t\cdot-:\Mat{\V}(Y,X)\to\Mat{\V}(Y,Z).
\end{align*}
preserve suprema and therefore have respective right adjoints
\begin{align*}
(-)\blackleft t:\Mat{\V}(X,Y)\to\Mat{\V}(Z,Y) &&\text{and}&&
t\blackright(-):\Mat{\V}(Y,Z)\to\Mat{\V}(Y,X).
\end{align*}
Here, for $\V$-relations $r:X\relto Y$ and $s:Y\relto Z$, 
\begin{align*}
 (r\blackleft t)(z,y)=\bigwedge_{x\in X}\hom(t(x,z),r(x,y))
 &&&
 (t\blackright s)(y,x)=\bigwedge_{z\in Z}\hom(t(x,z),s(y,z)).
\end{align*}
We call $r\blackleft t$ the \emph{extension of $r$ along $t$}, and $t\blackright s$ the \emph{lifting of $r$ along $t$}. We note here that, for $\V$-distributors $\varphi:A\modto X$, $\beta:Y\modto X$ and $\alpha:Z\modto Y$ where $\alpha$ is left adjoint, one easily establishes 
\begin{equation}\label{eq:ext_adj}
\varphi\blackright(\beta\cdot\alpha)=(\varphi\blackright\beta)\cdot\alpha;
\end{equation}
which actually holds in any quantaloid (see \cite[Lemma 1.8]{Hof11}, for instance).

\subsection*{II} The $\SET$-functor $T$ extends to a 2-functor $\Txi:\Mat{\V}\to\Mat{\V}$. To each $\V$-relation $r:X\times Y\to\V$, $\Txi$ assigns the $\V$-relation $\Txi r:TX\times TY \to\V$ such that, for every map $s:TX\times TY\to\V$,
 \[\xi\cdot Tr\le s\cdot \langle T\pi_1,T\pi_2\rangle \iff \Txi r\leq s:\]
\[
\xymatrix{TX\times TY\ar@{..>}[ddrr]^{\Txi r}&&\ar@{}[ddll]|-(0.75){\le}\\
&&\\
T(X\times Y)\ar[uu]^{\langle T\pi_1,T\pi_2\rangle}\ar[rr]_{\xi\cdot Tr} &&\V}
\]
In other words, regarding $TX$, $TY$ and $TX\times TY$ as discrete ordered sets, $\Txi r$ is the left Kan extension in $\ORD$ of $\xi\cdot Tr$ along $\langle T\pi_1,T\pi_2\rangle$. Hence, for $\fx\in TX$ and $\fy\in TY$,
\[
\Txi r(\fx,\fy)=\bigvee\left\{\xi\cdot Tr(\fw)\;\Bigl\lvert\;\fw\in T(X\times Y), T\pi_1(\fw)=\fx,T\pi_2(\fw)=\fy\right\}.
\]
The 2-functor $\Txi$ preserves the involution in the sense that $\Txi(r^\circ)=\Txi(r)^\circ$ (and we write $\Txi r^\circ$) for each $\V$-relation $r:X\relto Y$, $m$ becomes a natural transformation $m:\Txi\Txi\to\Txi$ and $e$ an op-lax natural transformation $e:1\to\Txi$, that is, $e_Y\cdot r\leq \Txi r\cdot e_X$ for all $r:X\relto Y$ in $\Mat{\V}$.

For $\Tth=\Uth_\two$, the extension above coincides with the one given in \citep{Bar70}; and for $\Tth=\Uth_{\Pp}$ and $\Tth=\Uth_{\Pm}$ one obtains
\[
 \Uxi r(\fx,\fy)=\sup_{A\in\fx,B\in\fy}\inf_{x\in A,y\in B} r(x,y)
\]
for all $r:X\relto Y$, $\fx\in UX$ and $\fy\in UY$ (see also \citep{CT03}).

Different methods for extending $\SET$-functors to $\REL$ can be found in \citep{Sea05,SS08,Sea09}.

\subsection*{III} $\V$-relations of the form $\alpha:TX\relto Y$, called \emph{$\Tth$-relations} and denoted by $\alpha:X\krelto Y$, will play an important role here. Given two $\Tth$-relations $\alpha:X\krelto Y$ and $\beta:Y\krelto Z$, their \emph{Kleisli convolution} $\beta\kleisli\alpha:X\krelto Z$ is defined as
\[\beta\kleisli\alpha=\beta\cdot \Txi\alpha\cdot m_X^\circ.\]
This operation is associative and has the $\Tth$-relation $e_X^\circ:X\krelto X$ as a lax identity:
\begin{align*}
 a\kleisli e_X^\circ&=a &\text{and}&& e_Y^\circ\kleisli a&\ge a,
\end{align*}
for any $a:X\krelto Y$.

\subsection*{IV} Those $\Tth$-relations satisfying the usual unit and composition axioms of a category define $\Tth$-categories: a \emph{$\Tth$-category} is a pair $(X,a)$ consisting of a set $X$ and a $\Tth$-relation $a:X\krelto X$ on $X$ such that
\begin{align*}
e_X^\circ&\le a &&\text{and}& a\kleisli a&\le a.
\end{align*}
Expressed elementwise, these conditions become
\begin{align*}
k&\le a(e_X(x),x) &&\text{and}& \Txi a(\fX,\fx)\otimes a(\fx,x)\le a(m_X(\fX),x)
\end{align*}
for all $\fX\in TTX$, $\fx\in TX$ and $x\in X$. We refer to the first condition as \emph{reflexivity} and to the second one as \emph{transitivity}. A function $f:X\to Y$ between $\Tth$-categories $(X,a)$ and $(Y,b)$ is a \emph{$\Tth$-functor} if $f\cdot a\le b\cdot Tf$, which in pointwise notation reads as
\[
 a(\fx,x)\le b(Tf(\fx),f(x))
\]
for all $\fx\in TX$, $x\in X$. The category of $\Tth$-categories and $\Tth$-functors is denoted by
\[\Cat{\Tth}.\]
If $\Tth=\Ith_\V$ is an identity theory, for a quantale $\V$, then a $\Tth$-category is just a $\V$-category and $\Tth$-functor means $\V$-functor (in the sense of \citep{EK66}). Therefore we write $\V$-category instead of $\Ith_\V$-category, $\V$-functor instead of $\Ith_\V$-functor, and
\[\Cat{\V}\]
instead of $\Cat{\Ith_\V}$. We also recall that $\Cat{\two}\simeq\ORD$, $\Cat{\Pp}\simeq\MET$ (the category of generalised metric spaces and non-expansive maps, see \citep{Law73}) and $\Cat{\Pm}\simeq\UMET$ (the category of generalised ultrametric spaces and non-expansive maps). Our principal examples are the ultrafilter theories $\Uth_\two$ and $\Uth_{\Pp}$: the main result of \citep{Bar70} states that $\Cat{\Uth_\two}$ is isomorphic to the category $\TOP$ of topological spaces and continuous maps, and in \citep{CH03} it is shown that $\Cat{\Uth_{\Pp}}$ is isomorphic to the category $\AP$ of approach spaces and non-expansive maps \citep{Low89} (regarding notation and results about approach spaces we refer to \citep{Low97}). The category $\Cat{\Uth_{\Pm}}$ can be identified with the full subcategory $\UAP$ of $\AP$ defined by all those approach spaces $(X,a)$ which satisfy
\[
 \max(\Uxi a(\fX,\fx),a(\fx,x))\ge a(m_X(\fX),x),
\]
for all $\fX\in UUX$, $\fx\in UX$ and $x\in X$. \emph{In the sequel we always refer to these presentations when talking about $\ORD$, $\MET$, $\UMET$, $\TOP$, $\AP$ or $\UAP$.}

\subsection*{V} The forgetful functor $\Cat{\Tth}\to\SET,\;(X,a)\mapsto X$ \emph{is topological}, hence it has a left and a right adjoint. In particular, the free $\Tth$-category on a set $X$ is given by $(X,e_X^\circ)$. The $\Tth$-category $G=(1,e_1^\circ)$ is a generator in $\Cat{\Tth}$. Furthermore, there is a canonical forgetful functor $\Cat{\Tth}\to\Cat{\V}$, commuting with the forgetful functors to $\SET$, which sends a $\Tth$-category $(X,a)$ to the $\V$-category $(X,a_0)$ where $a_0=a\cdot e_X$; and $\Cat{\Tth}\to\Cat{\V}$ has a concrete left adjoint which sends a $\V$-category $(X,c)$ to $(X,e_X^\circ\cdot\Txi c)$.

\subsection*{VI} A $\Tth$-relation $\varphi:X\krelto Y$ between $\Tth$-categories $X=(X,a)$ and $Y=(Y,b)$ is a \emph{$\Tth$-distributor}, denoted as $\varphi:X\kmodto Y$, if $\varphi\kleisli a\le\varphi$ and $b\kleisli \varphi\le \varphi$. Note that we always have $\varphi\kleisli a\ge\varphi$ and $b\kleisli \varphi\ge \varphi$, so that the $\Tth$-distributor conditions above are in fact equalities. $\Tth$-categories and $\Tth$-distributors form a 2-category, denoted by
\[
\Mod{\Tth},
\]
with Kleisli convolution as composition and with the 2-categorical structure inherited from $\Mat{\V}$. The identity in $\Mod{\Tth}$ on a $\Tth$-category $X=(X,a)$ is given by $a:X\kmodto X$. As before, we write
\[
 \Mod{\V}
\]
whenever $\Tth=\Ith_\V$ is an identity theory, and use $\varphi:X\modto Y$ instead of $\varphi:X\kmodto Y$ in this case.

\subsection*{VII}
\newcounter{subsect:seven}
\setcounter{subsect:seven}{7}
Each $\Tth$-functor $f:(X,a)\to(Y,b)$ induces an adjunction \[f_\Tast\dashv f^\Tast\] in $\Mod{\Tth}$, with $f_\Tast:X\kmodto Y$ and $f^\Tast:Y\kmodto X$ defined as $f_\Tast=b\cdot Tf$ and $f^\Tast=f^\circ\cdot b$ respectively. In fact, these assignments define functors
\begin{align*}
 (-)_\Tast:\Cat{\Tth}\to\Mod{\Tth} &&\text{and}&& (-)^\Tast:\Cat{\Tth}^\op\to\Mod{\Tth},
\end{align*}
where $X_\Tast=X=X^\Tast$. More generally, the definition of $f_\Tast$ and of $f^\Tast$ makes sense for any map $f:X\to Y$ between $\Tth$-categories, not just for $\Tth$-functors; but then $f_\Tast:X\krelto Y$ and $f^\Tast:Y\krelto X$ are in general only $\Tth$-relations. However:
\begin{equation}\label{eq:DistVsFun}
 \text{$f$ is a $\Tth$-functor}\iff\text{$f_\Tast$ is a $\Tth$-distributor}\iff\text{$f^\Tast$ is a $\Tth$-distributor.}
\end{equation}
A $\Tth$-functor $f:(X,a)\to(Y,b)$ is called \emph{fully faithful} if $f^\Tast\kleisli f_\Tast=1_X^\Tast$. Note that $f$ is fully faithful if and only if, for all $\fx\in TX$ and $x\in X$, $a(\fx,x)=b(Tf(\fx),f(x))$.

For a $\V$-functor $f:X\to Y$ we will, however, use the traditional notation $f_*:X\modto Y$ and $f^*:Y\modto X$. This distinction is convenient since at some occasions we will consider simultaneously the $\Tth$-distributor $f_\Tast:(X,a)\kmodto(Y,b)$ (induced by the $\Tth$-functor $f:(X,a)\to(Y,b)$) and the $\V$-distributor $f_*:(X,a_0)\modto(Y,b_0)$ (induced by the underlying $\V$-functor $f:(X,a_0)\to(Y,b_0)$).

The category \emph{$\Cat{\Tth}$ becomes a 2-category} by transporting the order-structure on hom-sets from $\Mod{\Tth}$ to $\Cat{\Tth}$ via the functor $(-)^\Tast:\Cat{\Tth}^\op\to\Mod{\Tth}$: for $\Tth$-functors $f,g:(X,a)\to(Y,b)$ we define (see \citep[Lemma 4.7]{HT10})
\begin{align*}
f\le g \text{ in $\Cat{\Tth}$ } &:\iff f^\Tast\le g^\Tast \text{ in $\Mod{\Tth}$} &\iff g_\Tast\le f_\Tast\text{ in $\Mod{\Tth}$}\\
 &\iff \forall x\in X\,.\,k\le b_0(f(x),g(x))\\
 &\iff f^*\le g^* \text{ in $\Mod{\V}$} &\iff g_*\le f_*\text{ in $\Mod{\V}$}.
\end{align*}
We call $f,g:X\to Y$ \emph{equivalent}, and write $f\simeq g$, if $f\le g$ and $g\le f$. Hence, $f\simeq g$ if and only if $f^\Tast=g^\Tast$ if and only if $f_\Tast=g_\Tast$. A $\Tth$-category $X$ is called \emph{separated} (see \citep{HT10} for details) whenever $f\simeq g$ implies $f=g$, for all $\Tth$-functors $f,g:Y\to X$ with codomain $X$. One easily verifies that it is enough to consider the case $Y=G$, so that $X$ is separated if and only if the ordered set $\Cat{\Tth}(G,X)$ is anti-symmetric. The full subcategory of $\Cat{\Tth}$ consisting of all separated $\Tth$-categories is denoted by
\[\CatSep{\Tth}.\]
\emph{The 2-categorical structure on $\Cat{\Tth}$ allows us to consider adjoint $\Tth$-functors}: a $\Tth$-functor $f:X\to Y$ is \emph{left adjoint} if there exists a $\Tth$-functor $g:Y\to X$ such that $1_X\le g\cdot f$ and $1_Y\ge f\cdot g$. Considering the corresponding $\Tth$-distributors, $f$ is left adjoint to $g$ in $\Cat{\Tth}$ if and only if $g_\Tast\dashv f_\Tast$ in $\Mod{\Tth}$, that is, if and only if $f_\Tast=g^\Tast$.

\subsection*{VIII} For a $\Tth$-distributor $\alpha:X\kmodto Y$, the composition function $-\kleisli\alpha$ has a right adjoint
\[
(-)\kleisli\alpha \dashv (-)\whiteleft\alpha
\]
where, for a given $\Tth$-distributor $\gamma:X\kmodto Z$, the \emph{extension} $\gamma\whiteleft\alpha:Y\kmodto Z$ is constructed in $\Mat{\V}$ as the extension $\gamma\whiteleft \alpha=\gamma\blackleft (\Txi\alpha\cdot m_X^\circ)$.
\[
\xymatrix{TX\ar[r]|-{\object@{|}}^\gamma\ar[d]|-{\object@{|}}_{m_X^\circ} & Z.\\
TTX\ar[d]|-{\object@{|}}_{\Txi\alpha}\\ TY\ar@{.>}[ruu]|-{\object@{|}}}
\]
Unfortunately, \emph{in general liftings need not exist} in $\Mod{\Tth}$ (see \citep[Example 1.7]{HS11}). 
\subsection*{IX}
The tensor product on $\V$ can be transported to $\Cat{\Tth}$ by putting \[(X,a)\otimes(Y,b)=(X\times Y,c),\] with
\[
c(\fw,(x,y))=a(T\pi_1(\fw),x)\otimes b(T\pi_2(\fw),y),
\]
where $\fw\in T(X\times Y)$, $x\in X$, $y\in Y$. The $\Tth$-category $E=(1,k)$ is a $\otimes$-neutral object, where $1$ is a singleton set and $k:T1\times 1\to\V$ the constant relation with value $k\in\V$. In general, this construction does not result in a closed structure on $\Cat{\Tth}$; however, it does so when defined in the larger category $\Gph{\Tth}$ of $\Tth$-graphs and $\Tth$-graph morphisms. Here a $\Tth$-graph (see \citep{CHT03}) is a pair $(X,a)$ consisting of a set $X$ and a $\V$-relation $a:TX\times X\to\V$ which is only required to satisfy
\[
 k\le a(e_X(x),x),
\]
for all $x\in X$; $\Tth$-graph morphisms are defined as $\Tth$-functors. There is an obvious full embedding
\[
 \Cat{\Tth}\hrw\Gph{\Tth}.
\]
For $\Tth$-graphs $X=(X,a)$ and $Y=(Y,b)$, $X\otimes Y$ is defined as above, but now $X\otimes-:\Gph{\Tth}\to\Gph{\Tth}$ has a right adjoint $(-)^X:\Gph{\Tth}\to\Gph{\Tth}$ (see \citep{Hof07}) where the structure $d$ on
\[
Y^X=\{f:X\to Y\mid \text{$f$ is a $\Tth$-functor of type $G\otimes X\to Y$}\}
\]
is given by
\[
d(\fp,h)=\bigwedge_{\substack{\fq\in T(Y^X\times X),x\in X\\ \fq\mapsto \fp}}\hspace{-4ex}
\hom(a(T\pi_2(\fq),x),b(T\!\ev(\fq),h(x))).
\]
Here $\ev$ denotes the evaluation map $\ev:Y^X\times X\to Y, (h,x)\mapsto h(x)$. The following result can be found in \citep{Hof07}.

\begin{proposition}\label{prop:TensExp}
Let $X=(X,a)$ be a $\Tth$-category with $a\cdot\Txi a=a\cdot m_X$. Then, for each $\Tth$-category $Y$, the structure $d$ on $Y^X$ is transitive. Hence, $X\otimes-:\Cat{\Tth}\to\Cat{\Tth}$ has a right adjoint $(-)^X:\Cat{\Tth}\to\Cat{\Tth}$. Moreover, for $h,h'\in Y^X$,
\[d(e_{Y^X}(h'),h)=\bigwedge_{x\in X}b(e_Y(h'(x)),h(x))).\]
\end{proposition}

\begin{definition}
A $\Tth$-category $X=(X,a)$ is called \emph{core-compact} whenever $a\cdot\Txi a=a\cdot m_X$.
\end{definition}

\begin{example}\label{ex:AlexSubbase}
The designation ``core-compact'' is motivated by the case of topological spaces. Classically, a topological space $X$ with topology $\calO$ is called core-compact whenever $x\in U\in\calO$ implies that there exists some $V\in\calO$ with $x\in V$ and $V$ is \emph{relatively compact} in $U$; the latter meaning that very open cover of $U$ contains a finite sub-cover of $V$, or, equivalently, every ultrafilter on $V$ has a convergence point in $U$. It is shown in \citep{Pis99} that $X$ is core-compact if and only if its convergence structure $a:UX\relto X$ satisfies $a\cdot \Uxi a=a\cdot m_X$. We find it worthwhile to note that the proof of the implication ``core-compact $\Rw$ $a\cdot \Uxi a=a\cdot m_X$'' can be adapted to subbases, under a certain condition. More in detail, for a set $X$ equipped with a subset $\calB$ of the powerset of $X$ (no axioms required), if $(X,\calB)$ is core-compact (defined as for topological spaces), then the induced convergence $a:UX\relto X$ (defined as for topological spaces) satisfies $a\cdot\Uxi a=a\cdot m_X$ provided that every ultrafilter has a smallest convergence point with respect to the convergence $a$ and the order relation $a\cdot e_X$. Since the topology induced by $\calB$ has the same convergence as $\calB$, one obtains a variation of Alexander's Sub-Base Lemma: A topological space where every ultrafilter has a smallest convergence point is core-compact if it is core-compact with respect to a sub-basis. We will apply this principle in Example \ref{ex:UpperVietoris}. We also note that this argument works for any property of topological spaces which can be equivalently expressed in terms of opens and in terms of ultrafilter convergence, without using the axioms of a topology; in particular in the classical case of compactness.
\end{example}

\section{Representable $\Tth$-categories and dualisation}\label{sect:ReprDual}


In \citep{CH09} we introduced a notion of dual $\Tth$-category as a crucial step towards the Yoneda lemma and related results. The basic idea is to associate to a $\Tth$-category $X$ a $\V$-category $MX$ which still contains all information about the $\Tth$-categorical structure of $X$, and then use the usual dualisation of $\V$-categories. Later, in \citep{Hof13,GH12}, we noted already that this construction is closely related to Nachbin's ordered compact Hausdorff spaces \citep{Nac50} as presented in \citep{Tho09}. In this section we continue this path and introduce a class of $\Tth$-categories (designated as representable $\Tth$-categories) which naturally admit a dual. 

Recall from \citep{Tho09} that the $\SET$-monad $\mT=\monad$ admits a natural extension to a monad on $\Cat{\V}$, in the sequel also denoted as $\mT=\monad$. Here the functor $T:\Cat{\V}\to\Cat{\V}$ sends a $\V$-category $(X,a_0)$ to $(TX,\Txi a_0)$, and with this definition $e_X:X\to TX$ and $m_X:TTX\to TX$ become $\V$-functors for each $\V$-category $X$. We also note that $T:\Cat{\V}\to\Cat{\V}$ is actually a 2-functor: if $f^*\le g^*$, then $(Tf)^*=\Txi(f^*)\le \Txi(g^*)=(Tg)^*$. 

Eilenberg--Moore algebras for this monad can be described as triples $(X,a_0,\alpha)$ where $(X,a_0)$ is a $\V$-category and $(X,\alpha)$ is an algebra for the $\SET$-monad $\mT$ such that $\alpha:T(X,a_0)\to(X,a_0)$ is a $\V$-functor. For $\mT$-algebras $(X,a_0,\alpha)$ and $(Y,b_0,\beta)$, a map $f:X\to Y$ is a homomorphism $f:(X,a_0,\alpha)\to(Y,b_0,\beta)$ precisely if $f$ preserves both structures, that is, whenever $f:(X,a_0)\to(Y,b_0)$ is a $\V$-functor and $f:(X,\alpha)\to(Y,\beta)$ is a $\mT$-homomorphism. Since the extension $\Txi$ of $T$ commutes with the involution $(-)^\circ$, with $(X,a_0,\alpha)$ also $(X,a_0^\circ,\alpha)$ is a $\mT$-algebra.

\begin{example}\label{ex:OrdMetComp}
It follows from Remark \ref{rem:homxi} that, for every topological theory $\Tth=\toptheory$, the internal $\hom$ in $\V$ combined with the $\mT$-algebra structure $\xi$ induces the Eilenberg--Moore algebra  $\V=(\V,\hom,\xi)$.

For $\Tth=\Uth_\two$, an algebra for the ultrafilter monad $\mU$ on $\ORD$ is an \emph{ordered compact Hausdorff space} as introduced in \citep{Nac50} (except that we do not assume anti-symmetry here). We recall that these ordered compact Hausdorff spaces are traditionally defined as triples $(X,\le,\calO)$ where $(X,\le)$ is an ordered set and $\calO$ is a compact Hausdorff topology on $X$ so that $\{(x,y)\mid x\le y\}$ is closed in $X\times X$, but the latter requirement means precisely that the convergence $\alpha:UX\to X$ of $\calO$ is monotone (see \citep{Tho09}). A trivial but important example of an ordered compact Hausdorff space is the two-element chain $\two=\{0,1\}$ with the discrete topology. We also note that, for an order relation $\le$ on $X$,
\[
 \fx\,(\Uxi\!\le)\,\fx'\iff \forall A\in\fx'\,.\,\downc A\in\fx,
\]
for all $\fx,\fx'\in UX$.

For $\Tth=\Uth_{\Pp}$ it seems natural to call an algebra for the ultrafilter monad $\mU$ on $\MET$ a \emph{metric compact Hausdorff space}. The set $[0,\infty]$ equipped with the metric $\hom(u,v)=v\trunminus u$ becomes a metric compact Hausdorff space $\Pp$ with the Euclidean compact Hausdorff topology whose convergence is given by $\xi:U[0,\infty]\to[0,\infty]$ (see Example \ref{ex:Theories} \eqref{ex:MetTh}). Similarly, for $\Tth=\Uth_{\Pm}$, we call an algebra for the ultrafilter monad $\mU$ on $\UMET$ an \emph{ultrametric compact Hausdorff space}, and $[0,\infty]$ becomes an ultrametric compact Hausdorff space where the metric is given by the internal hom of $\Pm$ (see Example \ref{ex:Theories} \eqref{ex:UMetTh}) and the topology is again the Euclidean compact Hausdorff topology.
\end{example}

There is a canonical functor
\[
K:(\Cat{\V})^\mT\to\Cat{\Tth}.
\]
which associates to each $X=(X,a_0,\alpha)$ in $(\Cat{\V})^\mT$ the $\Tth$-category $KX=(X,a)$ where $a=a_0\cdot\alpha$. Note that $(a_0\cdot\alpha)_0=a_0$, hence our notation remains consistent. The category $(\Cat{\V})^\mT$ is actually a 2-category with the order relation on hom-sets inherited from $\Cat{\V}$, and one easily verifies that $K$ is a 2-functor. Applying $K$ to $\V=(\V,\hom,\xi)$ produces the $\Tth$-category $\V=(\V,\hom_\xi)$ where
\[\hom_\xi:T\V\times\V\to\V,\;(\fv,v)\mapsto\hom(\xi(\fv),v).\]
We note that $\V=(\V,\hom_\xi)$ is separated since $\Cat{\Tth}(G,\V)\simeq\V$ in $\ORD$.

The functor $M:\Cat{\Tth}\to\Cat{\V}$ mentioned at the beginning of this section lifts to a functor $M:\Cat{\Tth}\to(\Cat{\V})^\mT$ sending $(X,a)$ to $(TX,\Txi a\cdot m_X^\circ,m_X)$. To see this, we have to show that
\[
 m_X:(TTX,\Txi(\Txi a\cdot m_X^\circ))\to(TX,\Txi a\cdot m_X^\circ)
\]
is a $\V$-functor. In fact, from $m_X\cdot m_{TX}=m_X\cdot Tm_X$ one obtains $m_{TX}\cdot Tm_X^\circ\le m_X^\circ\cdot m_X$ and then
\[
 m_X\cdot \Txi\Txi a\cdot Tm_X^\circ\le \Txi a\cdot m_{TX}\cdot Tm_X^\circ
\le \Txi a\cdot m_X^\circ\cdot m_X.
\]
Furthermore, one easily verifies that $M$ is a 2-functor. We denote from now on the ``original functor $M$'' going from $\Cat{\Tth}$ to $\Cat{\V}$ as $M_0:\Cat{\Tth}\to\Cat{\V}$.

\begin{examples}
For a topological space $X=(X,a)$, the order relation $\hat{a}=\Uxi a\cdot m_X^\circ$ is given by
\[
 \fx\,\hat{a}\,\fy\hspace{2em}\text{whenever }\overline{A}\in\fy\text{ for every }A\in\fx.
\]
For an approach space $X=(X,a)$, the metric $\hat{a}=\Uxi a\cdot m_X^\circ$ is given by
\[
\hat{a}(\fx,\fy)=\inf\{\eps\mid \forall A\in\fx\,.\,\overline{A}^{(\eps)}\in\fy\}.
\]
\end{examples}

\begin{theorem}
$M:\Cat{\Tth}\to(\Cat{\V})^\mT$ is left adjoint to $K:(\Cat{\V})^\mT\to\Cat{\Tth}$.
\end{theorem}
\begin{proof}
For every $\Tth$-category $X=(X,a)$, $e_X:X\to KM(X)$ is a $\Tth$-functor since
\[
 \Txi a\cdot m_X^\circ\cdot m_X\cdot Te_X=\Txi a\cdot m_X^\circ\ge\Txi a\cdot e_{TX}\ge e_X\cdot a,
\]
and we obtain a natural transformation $e:1\to KM$. Let now $X=(X,a_0,\alpha)$ be in $\Cat{\V}^\mT$. Then $\alpha$ is a $\mT$-algebra homomorphism $\alpha:(TX,m_X)\to (X,\alpha)$, and also a $\V$-functor $\alpha:(TX,\Txi(a_0\cdot\alpha)\cdot m_X^\circ)\to (X,a_0)$ since
\[
\alpha\cdot\Txi(a_0\cdot\alpha)\cdot m_X^\circ
=\alpha\cdot\Txi(a_0)\cdot T\alpha\cdot m_X^\circ
\le a_0\cdot\alpha\cdot T\alpha\cdot m_X^\circ
=a_0\cdot\alpha.
\]
Clearly, for every $f:X\to Y$ in $\Cat{\V}^\mT$ where $X=(X,a_0,\alpha)$ and $Y=(Y,a_0,\beta)$, the diagram
\[
\xymatrix{MK(X)\ar[r]^-{\alpha}\ar[d]_{Tf=MK(f)} & X\ar[d]^f\\ MK(Y)\ar[r]_-{\beta} & Y}
\]
commutes, hence the family $(\alpha)_{(X,a_0,\alpha)}$ is a natural transformation $MK\to 1$. Finally, for $(X,a)$ in $\Cat{\Tth}$ and $(Y,b_0,\beta)\in\Cat{\V}^\mT$,
\begin{align*}
 m_X\cdot Me_X=1_{MX} &&\text{and}&& \beta\cdot e_Y=1_Y;
\end{align*}
and the assertion follows.
\end{proof}

Via the adjunction $M\dashv K$ one obtains a lifting of the $\SET$-monad $\mT=\monad$ to a monad on $\Cat{\Tth}$, also denoted as $\mT=\monad$. Explicitly, $T:\Cat{\Tth}\to\Cat{\Tth}$ sends a $\Tth$-category $X=(X,a)$ to $(TX,\Txi a\cdot m_X^\circ\cdot m_X)$. Moreover, $e_X:X\to TX$ is fully faithful since $e_X^\circ\cdot \Txi a\cdot m_X^\circ\cdot m_X\cdot Te_X=a$.

\begin{proposition}
The monad $\mT=\monad$ on $\Cat{\Tth}$ is of Kock-Z\"oberlein type.
\end{proposition}
\begin{proof}
We show that $m_X\dashv e_{TX}$ in $\Cat{\Tth}$. By definition, $m_X\cdot e_{TX}=1_{TX}$. To see
$k\le s(\fX,e_{TX}\cdot m_X(\fX))$ for all $\fX\in TTX$, we show $m_X^\circ\cdot e_{TX}^\circ\cdot s\ge 1_{TTX}$ (where $s=\Txi(r\cdot m_X)\cdot m_{TX}^\circ$ and $r=\Txi a\cdot m_X^\circ$) in $\Mat{\V}$. In fact,
\[
 m_X^\circ\cdot e_{TX}^\circ\cdot \Txi(r\cdot m_X)\cdot m_{TX}^\circ
=m_X^\circ\cdot r\cdot m_X\ge m_X^\circ\cdot m_X\ge 1_{TTX}.\qedhere
\]
\end{proof}

Since the monad $\mT$ on $\Cat{\Tth}$ is of Kock-Z\"oberlein type, an algebra structure $\alpha:TX\to X$ on a $\Tth$-category $X$ is left adjoint to the unit $e_X:X\to TX$. However, unless $X$ is separated, a left adjoint $\alpha:TX\to X$ to $e_X$ is in general only a pseudo-algebra structure on $X$, that is, 
\begin{align}\label{eq:PseudoAlg}
 \alpha\cdot e_X&\simeq 1_X &\text{and}&& \alpha\cdot T\alpha&\simeq \alpha\cdot m_X.
\end{align}

\begin{definition}\label{def:Repr}
We call a $\Tth$-category $X$ \emph{representable} whenever $e_X:X\to TX$ has a left adjoint in $\Cat{\Tth}$. A $\Tth$-functor $f:X\to Y$ between representable $\Tth$-categories $X$ and $Y$, with left adjoint $\alpha:TX\to X$ and $\beta:TY\to Y$ respectively, is called a \emph{pseudo-homomorphism} whenever
\[
\beta\cdot Tf \simeq f\cdot\alpha.                                                                                                                                                                                                                                                                                                                                                                                                                                                                                                                                                                                                                                                     \]
\end{definition}
Of course, if $Y$ is separated, then one has equality above. We denote the category of representable $\Tth$-categories and pseudo-homomorphism by
\[
 \Repr{\Tth},
\]
and its full subcategory defined by the separated representable $\Tth$-categories by $\Repr{\Tth}_\sep$.

\begin{remark}
We borrowed the designation ``representable'' from \citep{Her00} where the notion of representable multicategory via a ``monadic 2-adjuntion between the 2-category of strict monoidal categories and that of multicategories'' is introduced. In a nutshell, strict monoidal categories are to multicategories what ordered compact Hausdorff spaces are to topological spaces.
\end{remark}

\begin{remark}\label{rem:AdjReprTalg}
For a separated representable $\Tth$-category $X=(X,a)$, the left adjoint $\alpha:TX\to X$ to $e_X:X\to TX$ is unique and actually the structure of a $\mT$-algebra on $X$. Therefore there is a canonical forgetful functor $\Repr{\Tth}_\sep\to\SET^\mT$ sending $(X,a)$ to $(X,\alpha)$ which is part of an adjunction
\[
 \Repr{\Tth}_\sep\adjunct{}{}\SET^\mT
\]
where the left adjoint $\SET^\mT\to\Repr{\Tth}_\sep$ interprets the $\mT$-structure $\alpha$ on a set $X$ as a $\Tth$-structure on $X$.
\end{remark}

\begin{proposition}\label{prop:CharReprCat}
The following assertions are equivalent, for a $\Tth$-category $X=(X,a)$.
\begin{eqcond}
\item $X$ is representable.
\item\label{cond:eqii} $X$ is core-compact and there is a map $\alpha:TX\to X$ such that $a=a_0\cdot\alpha$.
\end{eqcond}
\end{proposition}
\begin{proof}
Assume first that $X$ is representable. Then $e_X:X\to TX$ has a left adjoint $\alpha:TX\to X$ in $\Cat{\Tth}$ which necessarily satisfies \eqref{eq:PseudoAlg}. Hence also $\alpha\dashv e_X$ in $\Cat{\V}$, which gives (with $\hat{a}=\Txi\cdot m_X^\circ$)
\[
 a(\fx,x)=\hat{a}(\fx,e_X(x))=a_0(\alpha(\fx),x),
\]
and then we calculate $a_0\cdot\alpha\cdot m_X=a_0\cdot\alpha\cdot T\alpha\le a_0\cdot\alpha\cdot\Txi(a_0\cdot\alpha)$. Conversely, assume now \eqref{cond:eqii}. Then $e_X^\Tast=e_X^\circ\cdot\Txi a\cdot m_X^\circ\cdot m_X=a\cdot m_X$ and $\alpha_\Tast=a\cdot T\alpha=a\cdot\Txi a_0\cdot T\alpha=a\cdot\Txi a=a\cdot m_X$, hence $\alpha$ is a $\Tth$-functor and $\alpha\dashv e_X$ in $\Cat{\Tth}$.
\end{proof}

The following result is easy to prove.

\begin{lemma}\label{lem:MapsReprCats}
Let $(X,a)$, $(Y,b)$ be representable $\Tth$-categories with left adjoints $\alpha:TX\to X$ and $\beta:TY\to Y$ respectively, and let $f:X\to Y$ be a map. Then $f$ is a $\Tth$-functor $f:(X,a)\to(Y,b)$ if and only if $f:(X,a_0)\to(Y,b_0)$ is a $\V$-functor and $\beta\cdot Tf(\fx)\le f\cdot\alpha(\fx)$, for all $\fx\in TX$.
\end{lemma}

We have the comparison functor
\[
 K^\mT:(\Cat{\V})^\mT\to(\Cat{\Tth})^\mT
\]
which sends $(X,a_0,\alpha)$ in $(\Cat{\V})^\mT$ to $(X,a_0\cdot\alpha,\alpha)$. Furthermore, the forgetful functor $(-)_0:\Cat{\Tth}\to\Cat{\V}$ lifts to a functor $(-)_0^\mT:\Cat{\Tth}^\mT\to\Cat{\V}^\mT$ since the identity map $T(X_0)\to(TX)_0$ is a $\V$-functor, for each $\Tth$-category $X$. Clearly, $(K^\mT X)_0^\mT=X$ for every $X$ in $(\Cat{\V})^\mT$, and for $X=(X,a)$ in $(\Cat{\Tth})^\mT$ with Eilenberg--Moore structure $\alpha:TX\to X$ one has $a=a_0\cdot\alpha$ by Proposition \ref{prop:CharReprCat}. We conclude:

\begin{theorem}\label{thm:AlgVAlgT}
 $(\Cat{\V})^\mT\simeq (\Cat{\Tth})^\mT$.
\end{theorem}

The notion of pseudo-algebra was already lurking in the discussion above. In fact, just as for algebras, every pseudo-algebra structure $\alpha:TX\to X$ on a
$\V$-category $X=(X,a_0)$ gives rise to the representable $\Tth$-category $(X,a_0\cdot\alpha)$, and equivalent pseudo-algebra structures induce the same $\Tth$-category. Moreover, by Proposition \ref{prop:CharReprCat}, every representable $\Tth$-category is of this form.

Our next aim is to introduce a concept of dual $\Tth$-category which generalises the one for $\V$-categories. 

\begin{definition}\label{def:dualTGph}
A $\Tth$-graph $X=(X,a)$ is called \emph{dualisable} whenever $a_0=a\cdot e_X$ is transitive and $a=a_0\cdot\alpha$, for some map $\alpha:TX\to X$.
\end{definition}

Every $\Tth$-category $(X,a)$ where $a=a_0\cdot\alpha$ (for some map $\alpha:TX\to X$) is a dualisable $\Tth$-graph. Another important example will be provided by Lemma \ref{lem:VXdualisable}. For a dualisable $\Tth$-graph $X=(X,a)$, we write $X_0$ to denote its underlying $\V$-category $X_0=(X,a_0)$. We consider $TX$ as a discrete $\V$-category, so that $\alpha:TX\to X_0$ is a $\V$-functor. With this notation, $a_0\cdot\alpha=\alpha_*$ and, if $\alpha_*=a=\beta_*$, also $\alpha^*=\beta^*$ and therefore 
\[
 a_0^\circ\cdot\alpha=(\alpha^*)^\circ=(\beta^*)^\circ=a_0^\circ\cdot\beta.
\]

\begin{lemma}
Let $X=(X,a)$ be a dualisable $\Tth$-graph. Then $(X,a_0^\circ\cdot\alpha)$ is a dualisable $\Tth$-graph as well, and the underlying $\V$-category of $(X,a_0^\circ\cdot\alpha)$ is $(X_0)^\op$.
\end{lemma}
\begin{proof}
It suffices to show $a_0^\circ=a_0^\circ\cdot\alpha\cdot e_X$. From $a=a_0\cdot\alpha$ we infer $a_0=a_0\cdot\alpha\cdot e_X=(\alpha\cdot e_X)_*$, hence $a_0=(\alpha\cdot e_X)^*$ and therefore $a_0^\circ=a_0^\circ\cdot\alpha\cdot e_X$.
\end{proof}

\begin{definition}\label{def:dualOfTGph}
Let $X=(X,a)$ be a dualisable $\Tth$-graph. Then the \emph{dual} $\Tth$-graph $X^\op$ of $X$ is $X^\op=(X,a_0^\circ\cdot\alpha)$.
\end{definition}
By the discussion above, this definition is independent of the choice of $\alpha$. Of course, the dual of a $\V$-category in the sense above is just the usual dual. Also note that, even if $X$ is a $\Tth$-category, $X^\op$ need not be a $\Tth$-category (see Proposition \ref{prop:DualVsCore} below). 

The following result is a variation of Lemma \ref{lem:MapsReprCats}.

\begin{lemma}\label{lem:dualTVGraphFun}
For a $(\mT,\V)$-functor $f:(X,a)\to(Y,b)$ between dualisable $(\mT,\V)$-graphs with $a=a_0\cdot\alpha$ and $b=b_0\cdot\beta$, the map $f:X\to Y$ also defines a $(\mT,\V)$-functor $f^\op:X^\op\to Y^\op$ if and only if $f\cdot\alpha\simeq\beta\cdot Tf$.
\end{lemma}

In the sequel we will extend our terminology to $(\mT,\V)$-functor$f:(X,a)\to(Y,b)$ between dualisable $(\mT,\V)$-graphs and call $f$ a \emph{pseudo-homomorphism} if $f\cdot\alpha\simeq\beta\cdot Tf$.

\begin{proposition}\label{prop:DualVsCore}
Let $X=(X,a)$ be a $\Tth$-category where $a=a_0\cdot\alpha$, for some map $\alpha:TX\to X$. Then the following assertions are equivalent.
\begin{eqcond}
\item The $\Tth$-graph $X^\op$ is actually a $\Tth$-category.
\item $X$ is core-compact.
\item $X$ is representable.
\end{eqcond}
\end{proposition}
\begin{proof}
Proposition \ref{prop:CharReprCat} affirms (ii)$\Leftrightarrow$(iii), and (iii)$\Rightarrow$(i) is clear. Assume now that $X^\op$ is a $\Tth$-category. Since $X$ is a $\Tth$-category,
\[
 (\alpha\cdot T\alpha)_*=a_0\cdot\alpha\cdot T\alpha\le a_0\cdot\alpha\cdot \Txi a_0\cdot T\alpha\le a_0\cdot\alpha\cdot m_X=(a_0\cdot\alpha)_*;
\]
similarly, since $X^\op$ is a $\Tth$-category,
\[
 a_0^\circ\cdot\alpha\cdot T\alpha\le a_0^\circ\cdot\alpha\cdot m_X
\]
and therefore $(\alpha\cdot T\alpha)^*\le(a_0\cdot\alpha)^*$. Consequently, $(\alpha\cdot T\alpha)_*=(a_0\cdot\alpha)_*$, hence $a\cdot\Txi a=a\cdot m_X$.
\end{proof}

In conclusion, taking duals gives a functor $(-)^\op:\Repr{\Tth}\to\Repr{\Tth}$ which makes the diagram
\[
 \xymatrix{\Repr{\Tth}\ar[d]_{(-)_0}\ar[r]^{(-)^\op} & \Repr{\Tth}\ar[d]^{(-)_0}\\ \Cat{\V}\ar[r]_{(-)^\op} &  \Cat{\V}}
\]
commutative. 

By Propositions \ref{prop:TensExp} and \ref{prop:CharReprCat}, every representable $\Tth$-category is $\otimes$-exponentiable. It is interesting to observe that the canonical map $Y^{(X,a_0\cdot\alpha)}\to Y^{(X,\alpha)}$ is actually an embedding, for every $(X,a_0,\alpha)$ in $(\Cat{\V})^\mT$ (see \citep[Lemma 5.2]{Hof13} for a proof for the approach case, the general case is similar). For a $\Tth$-category $X$, its \emph{presheaf $\Tth$-category} $PX$ is defined as $PX:=\V^{(TX)^\op}$ with structure relation denoted as $\fspstrP{-}{-}$. By the observation above, this definition coincides with the one given in \citep{CH09}. Proposition \ref{prop:TensExp} implies that the underlying $\V$-category $(PX)_0$ is a full subcategory of the presheaf $\V$-category of $M_0(X)$, where, for $\psi,\psi'\in PX$,
\[
[\psi,\psi']:=\fspstrP{e_{PX}(\psi)}{\psi'}=\bigwedge_{\fx\in TX}\hom(\psi(\fx),\psi'(\fx)).
\]
A slight adaptation of \citep[Theorem 2.5]{CH09} gives
\begin{theorem}\label{CharTDist}
Let $\varphi:X\krelto Y$ be a $\Tth$-relation. The following assertions are equivalent.
\begin{eqcond}
\item $\varphi:X\kmodto Y$ is a $\Tth$-distributor.
\item $\varphi:(TX)^\op\otimes Y\to\V$ is a $\Tth$-functor.
\item $\mate{\varphi}:Y\to PX$ is a $\Tth$-functor.
\end{eqcond}
\end{theorem}

For each $\Tth$-category $X=(X,a)$, $a:X\kmodto X$ is a $\Tth$-distributor which gives us the \emph{Yoneda functor}
\[
 \yoneda_X=\mate{a}:X\to PX.
\]
We recall the following result from \citep{Hof11}.
\begin{theorem}\label{Yoneda}
Let $\psi:X\kmodto Z$ and $\varphi:X\kmodto Y$ be $\Tth$-distributors. Then, for all $\fz\in TZ$ and $y\in Y$,
\[
\fspstrP{T\mate{\psi}(\fz)}{\mate{\varphi}(y)}=(\varphi\whiteleft\psi)(\fz,y).
\]
In particular, for each $\psi\in PX$ and each $\fx\in TX$, $\psi(\fx)=\fspstrP{T\yoneda_X(\fx)}{\psi}$.
\end{theorem}

\subsection*{Example: topological (and approach) spaces}

Regarding $\Tth=\Uth_\two$, an ordered compact Hausdorff space $X=(X,\le,\alpha)$ induces a topological space $X$ by stipulating that an ultrafilter $\fx\in UX$ converges to $x\in X$ whenever $\alpha(\fx)\le x$, that is, by making $\alpha(\fx)$ the smallest convergence point of $\fx$. The ordered compact Hausdorff space $\two$ (see Example \ref{ex:OrdMetComp}) induces the Sierpi\'nski space $\two$ where $\{1\}$ is closed and $\{0\}$ is open, and $\two^\op$ has $\{1\}$ open and $\{0\}$ closed. Representable topological T$_0$-spaces (under the name \emph{stably compact spaces} or \emph{well-compact spaces}) are well studied, we refer to  \citep{Sim82,Jun04,Law11} for more information. Below and until the end of this section we develop some well-known basic properties of these spaces, mainly to connect the convergence-theoretic perspective of this paper with the classical account via open subsets.

By Proposition \ref{prop:CharReprCat}, every representable topological space $X=(X,a)$ satisfies $a\cdot \Uxi a=a\cdot m_X$, which is equivalent to $X$ being core-compact (see \citep{Pis99}). In fact, slighly more can be said:

\begin{lemma}
Every representable topological space is locally compact.
\end{lemma}
\begin{proof}
For every topological space $X$, the topology on $UX$ is generated by all sets of the form
\[
 A^\#=\{\fa\in UX\mid A\in\fa\},
\]
where $A\subseteq X$ is open. Furthermore, for any ultrafilter $\fX\in UUX$ with $A^\#\in\fX$, one has $m_X(\fX)\in A^\#$ and therefore $A^\#$ is compact; hence $UX$ is locally compact. If $X$ is representable, then $X$ is a split subobject of $UX$ (since $\alpha:UX\to X$ can be chosen so that $\alpha(e_X(x))=x$) and therefore also locally compact.
\end{proof}

Hence, a topological space $X$ is representable if and only if $X$ is locally compact and every ultrafilter on $X$ has a smallest convergence point (see Proposition \ref{prop:CharReprCat}). The latter condition says in particular that the set of limit points of an ultrafilter $\fx$ is irreducible. Since in any topological space an irreducible closed subset is the set of limit points of some ultrafilter, we find that $X$ is representable if and only if $X$ is locally compact, weakly sober (every irreducible closed subset is the closure of some point) and, for every $\fx\in UX$, the set of limit points of $\fx$ is irreducible. For any core-compact topological space $X$, the last condition is equivalent to stability of the way-below relation on the lattice of open subsets under finite intersections: $\bigcap_i U_i\ll\bigcap_i V_i$, for open subsets $U_1,\ldots,U_n$ and $V_1,\ldots,V_n$ ($n\in\N$) of $X$ with $U_i\ll V_i$ for each $1\le i\le n$ (see \citep{Sim82}). By definition, $U\ll V$ whenever every open cover of $V$ contains a finite subcover of $U$, which is the case if and only if every ultrafilter $\fx$ with $U\in\fx$ has a limit point in $V$. If $X$ is representable, $U\ll V$ if and only if any smallest limit point of an ultrafilter $\fx$ with $U\in\fx$ belongs to $V$. Hence:
\begin{proposition}
A topological space $X$ is representable if and only if $X$ is locally compact, weakly sober and the way-below relation on the lattice of opens is stable under finite intersection.
\end{proposition}
This stability condition on the way-below relation is sometimes replaced by a stability condition on the compact down-sets of $X$, as we explain next. Clearly, for $X$ representable, $\bigcap\varnothing=X$ is compact, and binary intersections of pairs of compact down-sets are compact: if $A,B\subseteq X$ are compact down-sets and $A\cap B\in\fx$, then any smallest convergence point of $\fx$ belongs to both $A$ and $B$ and therefore also to $A\cap B$. Secondly, since open subsets are down-closed, the down-closure (with respect to the underlying order) of a compact subset of a topological space is compact. Therefore, for a locally compact space $X$ and $U,V\subseteq X$ open,
\begin{equation}\label{eq:LocCompWayBelow}
 U\ll V\iff U\subseteq K\subseteq V\text{ for some compact down-set $K\subseteq X$.}
\end{equation}
From that one sees at once that stability of the way-below relation under finite intersection follows from stability of compact down-sets under finite intersection.
\begin{proposition}
A topological space $X$ is representable if and only if $X$ is locally compact, weakly sober and finite intersections of compact down-sets are compact.
\end{proposition}
By definition, a pseudo-homomorphism between representable topological spaces is a continuous map $f:X\to Y$ which preserves smallest convergence points of ultrafilters. 
\begin{proposition}\label{prop:SpecMap}
Let $f:X\to Y$ be a continuous map between representable topological spaces. Then the following assertions are equivalent.
\begin{eqcond}
\item\label{cond:spec1} $f$ is a pseudo-homomorphism.
\item\label{cond:spec2} For every compact down-set $K\subseteq Y$, $f^{-1}(K)$ is compact.
\item\label{cond:spec3} The frame homomorphism $f^{-1}:\calO Y\to\calO X$ preserves the way-below relation.
\end{eqcond}
\end{proposition}
\begin{proof}
Cleary, \eqref{cond:spec1}$\Rw$\eqref{cond:spec2}; and the implication \eqref{cond:spec2}$\Rw$\eqref{cond:spec3} follows from \eqref{eq:LocCompWayBelow}. Assume now \eqref{cond:spec3} and let $x\in X$ be a smallest convergence point of $\fx\in UX$. Assume that $Uf(\fx)\to y\in Y$. Let $U,V\subseteq Y$ be open subsets with $y\in U\ll V$. Then $f^{-1}(U)\ll f^{-1}(V)$ and $f^{-1}(U)\in\fx$, hence $x\in f^{-1}(V)$ and therefore $f(x)\in V$. We conclude that $f(x)\le y$.
\end{proof}
A continuous map $f:X\to Y$ satisfying condition \eqref{cond:spec2} (and hence also \eqref{cond:spec1} and \eqref{cond:spec3}) above is called \emph{spectral}. 
\begin{corollary}
Let $X$ be a representable topological space. A continuous map $\varphi:X\to\two$ is a homomorphism if and only if the open set $\varphi^{-1}(0)$ is compact.
\end{corollary}

From Lemma \ref{lem:MapsReprCats} we obtain

\begin{proposition}
Let $(X,\le,\alpha)$ be an ordered compact Hausdorff space and let $a=\le\cdot\alpha$ be the induced topology. A subset $A\subseteq X$ is open in $(X,a)$ if and only if $A$ is down-closed and open in the compact Hausdorff space $(X,\alpha)$.
\end{proposition}

\begin{proposition}
Let $X$ be a representable space, $\fx\in UX$ and $x_0\in X$ be a smallest convergence point of $\fx$. For any $x\in X$, $x\le x_0$ if and only if $\fx$ contains all complements of compact down-sets $B$ with $x\notin B$.
\end{proposition}
\begin{proof}
If $x\le x_0$, then $\fx$ cannot contain any compact down-sets $B$ with $x\notin B$. Assume now that $\fx$ contains these subsets. Take a neighbourhood $B$ of $x_0$ where $B$ is a compact down-set. Then $x\in B$ since otherwise $B\in\fx$ and $X\setminus B\in\fx$. 
\end{proof}

\begin{corollary}\label{cor:TopDual}
Let $X$ be a representable space. Then the topology of $X^\op$ is generated by the complements of compact down-sets $B$ of $X$. Furthermore, the ultrafilter convergence of the topology generated by the opens and the complements of compact down-sets of $X$ is given by taking smallest convergence points of ultrafilters of $X$.
\end{corollary}

We note that this notion of dual space was introduced by M. Hochster (see \citep{Hoc69}).

\begin{corollary}
Let $(X,\le,\alpha)$ be an anti-symmetric ordered compact Hausdorff space. Then the topology of $(X,\alpha)$ is generated by the open subsets and the complements of compact down-sets of the representable space $(X,\le\cdot\alpha)$.
\end{corollary}

It is ``folklore'' that the category of anti-symmetric ordered compact Hausdorff spaces and homomorphisms is equivalent to the category $\STCOMP$ of stably compact spaces and spectral maps (the first appearance of this result seems to be \citep{GHK+80}), which is the restriction of Theorem \ref{thm:AlgVAlgT} to separated spaces. A stably compact space $X$ is called \emph{spectral} whenever the compact opens form a basis for the topology (which is equivalent to the statement that the source $(\varphi:X\to\two)$ of all homomorphisms into $\two$ is point-separating and initial in $\TOP$, and also in the category of stably compact spaces and spectral maps). For a spectral space $X$ and $U,V\subseteq X$ open, $U\ll V$ if and only if $U\subseteq K\subseteq V$  for some compact open subset $K\subseteq X$; hence a continuous map $f:Y\to X$ (where $Y$ is representable) is spectral if and only if the inverse image of every compact open subset of $X$ is compact in $Y$. A famous result of M.H. Stone \citep{Sto38} states that the category of spectral spaces and spectral maps is dually equivalent to the category $\DLAT$ of distributive lattices and homomorphisms. A different perspective on this duality was given in \citep{Pri70}: $\DLAT$ is also dually equivalent to the category of (nowadays called) \emph{Priestley spaces} and homomorphisms. Here a Priestley spaces is an anti-symmetric ordered compact Hausdorff space where $x\not\le y$ implies the existence of an clopen down-set $V$ and a clopen up-set $U$ with $x\in U$, $y\in V$ and $U\cap V=\varnothing$ (equivalently: the source $(\varphi:X\to\two)$ of all homomorphisms into $\two$ is point-separating and initial in the category of ordered compact Hausdorff spaces and homomorphisms). In particular, both results together imply the equivalence between spectral spaces and Priestley spaces which is a restriction of the aforementioned equivalence between stably compact spaces and anti-symmetric ordered compact Hausdorff spaces. We also note that, for $X$ compact Hausdorff, $X$ is spectral if and only if the simultaneously closed and open subsets of $X$ form a basis for the topology of $X$, i.e.\ if $X$ is a \emph{Stone space}. Every continuous map between Stone spaces is spectral, and the full subcategory $\STONE$ of $\SPEC$ defined by all Stone spaces is dually equivalent to the category $\BOOL$ of Boolean algebras and homomorphisms (\citep{Sto36,Joh86}).

The case of metric compact Hausdorff spaces (we consider now $\Tth=\Uth_{\Pp}$) was studied in \citep{GH12}). An approach space $X=(X,a)$ is representable if and only if $X$ is weakly sober, $+$-exponentiable and has the property that $a(\fx,-)$ is an approach prime element, for every $\fx\in UX$ (we refer to \citep{BLO06} and \citep{Olm05} for the theory of sober approach space). The metric compact Hausdorff space $\Pp$ (see Example \ref{ex:OrdMetComp}) induces the ``Sierpi\'nski approach space'' $\Pp$ with approach convergence structure $\lambda(\fx,x)=x\trunminus\xi(\fx)$; but, in contrast to the topological case, $\Pp$ is not isomorphic to $\Pp^\op$ (for instance, $\Pp$ is injective but $\Pp^\op$ is not). Similarly, the ultrametric compact Hausdorff space $\Pm$ produces the approach space $\Pm$. We note that both $\Pp$ and $\Pm$ have the same underlying topological space.

\section{Cocomplete $\Tth$-categories}\label{sect:CoCts}

By an appropriate translation from the $\V$ to the $\Tth$-case one can transport the notions of weighted colimit (see \citep{EK66,Kel82}) and cocompleteness into the realm $\Tth$-categories, as we recall now briefly from \citep{Hof11} and \citep{CH09a}. A weighted colimit diagram in a $\Tth$-category $X$ is given by a $\Tth$-functor $d:D\to X$ and a $\Tth$-distributor $\varphi:D\kmodto G$ (where $G=(1,e_1^\circ)$).
\[
 \xymatrix{D\ar@{-_{>}}|-{\object@{o}}[d]_\varphi\ar[r]^d & X\\ G}
\]
A \emph{colimit} of such weighted diagram is an element $x\in X$ which represents $d_\Tast\whiteleft \varphi$, that is, $x_\Tast=d_\Tast\whiteleft \varphi$. If such $x$ exists, it is unique up to equivalence, and one calls $x$ a \emph{$\varphi$-weighted colimit of $d$} and writes $x\simeq \colim (d,\varphi)$. We say that a $\Tth$-functor $f:X\to Y$ \emph{preserves} the $\varphi$-weighted colimit $x$ of $d$ if $f(x)$ is the $\varphi$-weighted colimit of $f\cdot d$, that is, if $f(x)_\Tast=(f\cdot d)_\Tast\whiteleft \varphi$. A $\Tth$-functor $f:X\to Y$ is called \emph{cocontinuous} if it preserves all weighted colimits which exist in $X$, and a $\Tth$-category $X$ is \emph{cocomplete} if every weighted colimit diagram in $X$ has a colimit in $X$. As in the $\V$-category case, cocompleteness of $X$ follows from the existence of colimits along identities. In fact, for any weight $\varphi:D\modto G$, the $\varphi$-weighted colimit of $d$ exists if and only if the $(\varphi\kleisli d^\Tast)$-weighted colimit of $1_X:X\to X$ exists, and in that case one has $\colim(d,\varphi)\simeq\colim(1_X,\varphi\kleisli d^\Tast)$. Moreover, a $\Tth$-functor $f:X\to Y$ preserves the $\varphi$-weighted colimit of $d$ if and only if it preserves the $(\varphi\kleisli d^\Tast)$-weighted colimit of $1_X$. In the sequel we will write $\Sup_X(\psi)$ (or simply $\Sup(\psi)$) instead of $\colim(1_X,\psi)$.

For a cocomplete $\Tth$-category $X$, the map $\Sup_X:PX\to X$ turns out to be left adjoint to the Yoneda embedding $\yoneda_X:X\to PX$ in $\Cat{\V}$; however, $\Sup_X$ is in general not a $\Tth$-functor (see \cite[Example 5.7]{HW11}). A $\Tth$-category $X$ is called \emph{totally cocomplete} whenever $\yoneda_X:X\to PX$ has a left adjoint $\Sup_X:PX\to X$ in $\Cat{\Tth}$. Curiously, total cocompleteness can be characterised by the existence of a slighly more general type of colimits, as we explain next. From now on we let in a \emph{weighted colimit diagram the weight $\varphi:D\kmodto A$ be an arbitrary $\Tth$-distributor}. A colimit of such a diagram is a $\Tth$-functor $g:A\to X$ which represents $d_\Tast\whiteleft \varphi$ in the sense that $g_\Tast=d_\Tast\whiteleft \varphi$, and we write $g\simeq \colim (d,\varphi)$. We note that one still has $\colim(d,\varphi)\simeq\colim(1_X,\varphi\kleisli d^\Tast)$. A $\Tth$-functor $f:X\to Y$ preserves the $\varphi$-weighted colimit $g$ of $d$ if $f\cdot g$ is the $\varphi$-weighted colimit of $f\cdot d$, that is, if $(f\cdot g)_\Tast=(f\cdot d)_\Tast\whiteleft \varphi$. We write
\[
 \Cocts{\Tth}
\]
to denote the category of totally cocomplete $\Tth$-categories and weighted colimit preserving $\Tth$-functors, and $\Cocts{\Tth}_\sep$ for its full subcategory defined by the separated $\Tth$-categories. As before, in the $\V$-case we use the designations $\Cocts{\V}$ and $\Cocts{\V}_\sep$.

For every $\Tth$-distributor $\varphi:X\kmodto Y$, the function $-\kleisli\varphi:PY\to PX$ is actually a $\Tth$-functor $P\varphi:PY\to PX$, and this construction yields a functor $P:\Mod{\Tth}^\op\to\Cat{\Tth}$. In fact:
\begin{theorem}
The functor $P:\Mod{\Tth}^\op\to\Cat{\Tth}$ is right adjoint to $(-)^\Tast:\Cat{\Tth}\to\Mod{\Tth}^\op$. The units of this adjunction are given by $\yoneda_X:X\to PX$ and $(\yoneda_X)_\Tast:X\kmodto PX$ respectively. The induced monad $\mP=\pmonad$ on $\Cat{\Tth}$ is of Kock-Z\"oberlein type (here $\yonmult_X=-\kleisli(\yoneda_X)_\Tast$).
\end{theorem}

\begin{theorem}
The following assertions are equivalent, for a $\Tth$-category $X$.
\begin{eqcond}
\item $X$ is injective w.r.t.\ fully faithful $\Tth$-functors.
\item $\yoneda_X:X\to PX$ has a left inverse $\Sup_X:PX\to X$, that is, $\Sup_X\cdot\yoneda_X\simeq 1_X$.
\item $\yoneda_X:X\to PX$ has a left adjoint $\Sup_X:PX\to X$.
\item $X$ has all weighted colimits (in the generalised sense).
\end{eqcond}
\end{theorem}
Here a $\Tth$-category $X$ is called \emph{injective} if, for all $\Tth$-functors $f:A\to X$ and fully faithful $\Tth$-functors $i:A\to B$, there exists a $\Tth$-functor $g:B\to X$ such that $g\cdot i\simeq f$. Clearly, for a separated $\Tth$-category $X$ we have then $g\cdot i= f$.
\begin{remark}
In the proof of (ii)$\Rw$(iii) one shows that any left inverse of $\yoneda_X:X\to PX$ is actually left adjoint to $\yoneda_X$. Then, given a left adjoint $\Sup_X:PX\to X$ of $\yoneda_X$, the colimit of a diagram defined by $d:D\to X$ and $\varphi:D\kmodto A$ can be calculated as $\Sup_X\cdot Pd\cdot\mate{\varphi}$.
\end{remark}
\begin{corollary}
For each $\Tth$-category $X$, $PX$ is cocomplete where $\Sup_{PX}=-\kleisli(\yoneda_X)_\Tast$.
\end{corollary}
\begin{proposition}
Let $f:X\to Y$ be a $\Tth$-functor between totally cocomplete $\Tth$-categories. Then the following assertions are equivalent.
\begin{eqcond}
\item $f$ preserves all weighted colimits (in the generalised sense).
\item $f$ preserves all weighted colimits with weight of type $D\kmodto G$.
\item The diagram
\[
 \xymatrix{PX\ar[r]^{Pf}\ar[d]_{\Sup_X}\ar@{}[dr]|\simeq & PY\ar[d]^{\Sup_Y}\\ X\ar[r]_f & Y}
\]
commutes up to equivalence. 
\end{eqcond}
\end{proposition}
\begin{theorem}
The category $\Cocts{\Tth}_\sep$ is precisely the category $(\Cat{\Tth})^\mP$ of Eilenberg--Moore algebras for $\mP$. Moreover, the canonical forgetful functors $\Cocts{\Tth}_\sep\to\Cat{\V}$ and $\Cocts{\Tth}_\sep\to\SET$ are both monadic.
\end{theorem}

\begin{proposition}
Every left adjoint $\Tth$-functor is cocontinuous. A $\Tth$-functor between cocomplete $\Tth$-categories is left adjoint if and only if it is cocontinuous.
\end{proposition}

For a $\Tth$-category $X=(X,a)$, one has the $\V$-category structure $\hat{a}:=\Txi a\cdot m_X^\circ:TX\relto TX$ on $TX$, which is indeed a $\Tth$-distributor $\hat{a}:X\kmodto TX$ since
\begin{align*}
& \hat{a}\cdot m_X\cdot \Txi\hat{a}\cdot m_X^\circ\le\hat{a}\cdot\hat{a}\cdot m_X\cdot m_X^\circ=\hat{a},\\
& \hat{a}\cdot \Txi a\cdot m_X^\circ=\hat{a}\cdot\hat{a}=\hat{a}.
\end{align*}
Therefore one obtains a $\Tth$-functor $\yonedaT_X:TX\to PX$, and one easily verifies $\yonedaT_X\cdot e_X=\yoneda_X$. Consequently:
\begin{proposition}\label{prop:TotCoComplRepr}
Every totally cocomplete $\Tth$-category is representable. Moreover, every left adjoint $\Tth$-functor between representable $\Tth$-categories is a pseudo-homomorphism.
\end{proposition}

\begin{remark}\label{rem:CoctsOverSetT}
From the proposition above we obtain a forgetful functor $\Cocts{\Tth}\to\Repr{\Tth}$ which restricts to separated objects. Therefore we also have functors (see Remark \ref{rem:AdjReprTalg})
\[
 \Cocts{\Tth}_\sep\to\Repr{\Tth}_\sep\to\SET^\mT,
\]
which commute with the canonical forgetful functors to $\SET$. The composite $\Cocts{\Tth}_\sep\to\SET^\mT$ is even monadic since both $\Cocts{\Tth}_\sep$ and $\SET^\mT$ are monadic over $\SET$ (see \citep{Lin69}).
\end{remark}

\section{The Vietoris monad}\label{sect:Vietoris}

\subsection*{General assumption}
\emph{From now on until the end of this paper, $\Tth=\toptheory$ denotes a (strict) topological theory where, moreover, $T1=1$.}

\bigskip
Under these conditions, the $\Tth$-category $\V=(\V,\homV)\simeq P1$ is totally cocomplete. Furthermore, since $e_1:1\to T1$ is a bijection, we can identify a $\V$-relation $\varphi:1\relto X$ with the $\Tth$-relation $\varphi\cdot e_1^\circ:1\krelto X$. If, moreover, $X=(X,a)$ is a $\Tth$-category, then $\varphi\cdot e_1^\circ$ is a $\Tth$-distributor of type $G\kmodto X$ if and only if
\[
 a\cdot\Txi\varphi\cdot e_1\le \varphi.
\]
Note that $a\cdot\Txi\varphi\cdot e_1\ge \varphi$ holds for every $\V$-relation $\varphi:1\relto X$.

\begin{lemma}\label{lem:VXdualisable}
For every $\Tth$-category $X=(X,a)$, the $\Tth$-graph $\V^X$ is dualisable.
\end{lemma}
\begin{proof}
From Proposition \ref{prop:TensExp} we know that the underlying $\V$-graph structure of $\V^X$ is transitive. Furthermore,
\begin{align*}
 \V^X(\fp,h)&=
\bigwedge_{x\in X}\bigwedge_{\fx\in TX}\bigwedge_{\substack{\fw\in T(\V^X\times X)\\T\pi_1(\fw)=\fp,T\pi_2(\fw)=\fx}}\hom(a(\fx,x),\hom(\xi\cdot T\ev(\fw),h(x)))\\
&=\bigwedge_{x\in X}\bigwedge_{\fx\in TX}\bigwedge_{\substack{\fw\in T(\V^X\times X)\\T\pi_1(\fw)=\fp,T\pi_2(\fw)=\fx}}\hom(a(\fx,x)\otimes\xi\cdot T\ev(\fw),h(x))\\
&=\bigwedge_{x\in X}\hom(\bigvee_{\fx\in TX}a(\fx,x)\otimes\bigvee_{\substack{\fw\in T(\V^X\times X)\\T\pi_1(\fw)=\fp,T\pi_2(\fw)=\fx}}\xi\cdot T\ev(\fw),h(x))\\
&=\bigwedge_{x\in X}\hom(a\cdot\Txi\ev(\fp,x),h(x))\\
&=[a\cdot\Txi\ev(\fp,-),h],
\end{align*}
where in the last line we consider $\ev:\V^X\times X\to\V$ as a $\V$-relation $\ev:\V^X\relto X$. Finally, writing $i_\fp:1\to T(\V^X)$ for the mapping sending the unique point of $1$ to $\fp\in T(\V^X)$, the composite
\[
 1\xrightarrow{\,i_\fp\,}T(\V^X)\xrelto{\,\Txi\ev\,}TX\xrelto{\,a\,}X
\]
is a $\Tth$-distributor of type $G\kmodto X$ since 
\[
 a\cdot\Txi a\cdot\Txi\Txi\ev\cdot Ti_\fp\cdot e_1\le a\cdot\Txi a\cdot m_X^\circ\cdot\Txi\ev\cdot i_\fp=a\cdot \Txi\ev\cdot i_\fp.
\]
Therefore $a\cdot\Txi\ev(\fp,-)$ belongs to $\V^X$.
\end{proof}

In the sequel we denote the composite $\V$-relation $a\cdot\Txi\ev$ by $\mu:T(\V^X)\relto X$.

\begin{corollary}\label{cor:reprspaceVpowerX}
For each core-compact $\Tth$-category $X=(X,a)$, $\V^X$ is a separated representable $\Tth$-category where the left adjoint of the $\Tth$-functor $e_{\V^X}:\V^X\to T(\V^X)$ is given by
\[
 \mate{\mu}:T(\V^X)\to\V^X,\;\fp\mapsto a\cdot\Txi\ev(\fp,-).
\]
\end{corollary}
\begin{proof}
For $X$ core-compact, $\V^X$ is separated and injective since $\V$ is, hence $\V^X$ is totally cocomplete and therefore representable (see Proposition \ref{prop:TotCoComplRepr}).
\end{proof}

\begin{example}\label{ex:UpperVietoris}
We consider $\Tth=\Uth_\two$, and let $X$ be a topological space. We write $\calO$ for the collection of all open subsets of $X$, and $\calO(x)$ for the set of all open neighbourhoods of $x\in X$. We can identify $2^X$ with the set of all closed subsets of $X$. For any subset $V\subseteq X$, we put
\[
 V^\Diamond=\{A\in \two^X\mid A\cap V\neq\varnothing\}.
\]
For an ultrafilter $\fp$ on $\two^X$, the smallest convergence point $\mu(\fp)$ of $\fp$ can be calculated as
\begin{align*}
\mu(\fp) &= \{x\in X\mid \exists \fx\in UX\,.\,\fp\,(\Uxi\ev)\,\fx\,\&\,\fx\to x\}\\
 &= \{x\in X\mid \forall V\in\calO(x),\calA\in\fp\,\exists A\in\calA,y\in V\,.\,y\in A\}\\
 &= \{x\in X\mid \forall V\in\calO(x),\calA\in\fp\,.\,V^\Diamond\cap\calA\neq\varnothing\}\\
&= \{x\in X\mid \forall V\in\calO(x)\,.\,V^\Diamond\in\fp\}.
\end{align*}
Therefore, for any $A\in(\two^X)^\op$,
\begin{align*}
 \fp\to A &\iff A\subseteq\mu(\fp)\\
&\iff \forall x\in A,V\in\calO(x)\,.\,V^\Diamond\in\fp\\
&\iff \forall V\in\calO\,.\,(V\cap A\neq\varnothing\;\Rw\; V^\Diamond\in\fp)\\
&\iff \forall V\in\calO\,.\,(A\in V^\Diamond\;\Rw\; V^\Diamond\in\fp);
\end{align*}
hence the convergence of the pseudo-topological space $(2^X)^\op$ is induced by $\{V^\Diamond\mid V\in\calO\}$ and therefore $VX:=(2^X)^\op$ is actually a topological space. This topology on the set of closed subsets of a topological space is known as the \emph{lower Vietoris topology} (see \citep{CT97}, for instance). We find it remarkable that, albeit $2^X$ belongs to $\TOP$ if and only if $X$ is exponentiable (see \citep{Sch84}), its dual $(2^X)^\op$ belongs always to $\TOP$. In fact, we can now easily derive the well-known characterisation of exponentiable spaces as precisely the core-compact ones (see \citep{DK70} and \citep{Isb75,Isb86}):
\begin{align*}
 \text{$X$ is exponentiable} &\iff \text{$(VX)^\op$ is topological}\\
&\iff \text{$VX$ is core-compact}&&\text{[Proposition \ref{prop:DualVsCore}]}\\
&\iff \text{$X$ is core-compact.}
\end{align*}
The last equivalence follows from the ``Sub-Base Lemma'' of Example \ref{ex:AlexSubbase} (applied to the sub-base $\{V^\Diamond\mid V\in\calO\}$ of $VX$).

We also note that $V^\Diamond\cap\calA\neq\varnothing$ is equivalent to $V\cap\bigcup\calA\neq\varnothing$, and therefore
\begin{align*}
 \mu(\fp) &= \{x\in X\mid \forall\calA\in\fp, V\in\calO(x)\,.\,V\cap\bigcup\calA\neq\varnothing\}\\
&= \{x\in X\mid \forall\calA\in\fp\,.\,x\in\overline{\bigcup\calA}\}\\
&=\bigcap_{\calA\in\fp}\overline{\bigcup\calA}.
\end{align*}
Here $\overline{(-)}$ denotes the closure of the topological space $X$. For $K\subseteq X$ compact, $K^\Diamond$ is a compact down-set in $VX$ and therefore its complement is open in $(VX)^\op$. Furthermore, for $X$ locally compact, one easily verifies that the sets
\[
 (K^\Diamond)^\complement=\{A\in VX\mid A\cap K=\varnothing\}\hspace{3em}\text{($K\subseteq X$ compact)}
\]
generate the convergence of $VX^\op$ (defined by $\fp\to A\iff \mu(\fp)\subseteq A$), which confirms that the topology of $\two^X\simeq(VX)^\op$ is the compact-open topology.
\end{example}

\begin{proposition}
For each $\Tth$-category $X$, the $\Tth$-graph $(\V^X)^\op$ is a $\Tth$-category.
\end{proposition}
\begin{proof}
Let $X=(X,a)$ be a $\Tth$-category. We write $c:T(\V^X)\relto \V^X$ for the $\Tth$-graph structure of $(\V^X)^\op$, by definition:
\[
 c(\fp,h)=[h,\mu(\fp,-)]=\bigwedge_{x\in X}\hom(h(x),\mu(\fp,x)),
\]
for all $\fp\in T(\V^X)$ and $h\in\V^X$. Hence, $c=\ev\blackright\mu$ in $\Mat{\V}$, where $\ev:\V^X\relto\V$ and $\mu:T(\V^X)\relto X$. Since $m_{\V^X}:TT(\V^X)\to T(\V^X)$ is left adjoint in $\Mat{\V}$, it follows that $c\cdot m_{\V^X}=\ev\blackright(\mu\cdot m_{\V^X})$ (see \eqref{eq:ext_adj} in Subsection \Roman{subsect:one} of Section \ref{sect:setting}). To conclude $c\cdot\Txi c\le c\cdot m_{\V^X}$, we show $\ev\cdot c\cdot\Txi c\le \mu\cdot m_{\V^X}$; and to see this, we calculate:
\begin{multline*}
 \ev\cdot c\cdot\Txi c\le \mu \cdot\Txi c =a\cdot\Txi\ev\cdot\Txi c=a\cdot\Txi\mu=a\cdot\Txi a\cdot\Txi\Txi\ev\\
\le a\cdot m_X\cdot \Txi\Txi\ev=a\cdot\Txi\ev\cdot m_{\V^X}=\mu\cdot m_{\V^X}.
\qedhere
\end{multline*}
\end{proof}

We put $VX:=(\V^X)^\op$, and denote the structure on $VX$ by $\fspstrV{-}{-}$; and the underlying $\V$-category structure by $\langle-,-\rangle$. Hence, for $\fp\in TVX$ and $\varphi\in VX$,
\[
 \fspstrV{\fp}{\varphi}=\langle\mate{\mu}(\fp),\varphi\rangle=[\varphi,\mate{\mu}(\fp)].
\]

\begin{example}\label{ex:topappV}
For both $\Tth=\Uth_{\Pp}$ and $\Tth=\Uth_{\Pm}$ and a topological space $X$ (viewed as an approach space), the underlying set of the approach space $VX$ is the set of all lower semi-continuous functions from $X$ to $[0,\infty]$ (see \citep[Proposition 2.1.8]{Low97}).
\end{example}

Albeit liftings of $\Tth$-distributors do not exist in general in $\Mod{\Tth}$, it is shown in \citep{HW11} that $\Mod{\Tth}$ admits liftings of $\Tth$-distributors along $\Tth$-distributors of type $1\kmodto X$.

\begin{lemma}
\label{lemma:Tlifting}
For all $\Tth$-distributors $\varphi: Y\kmodto X$ and $\psi: G\kmodto X$, $\varphi$ has a lifting $\psi\whiteright \varphi$ along an $\psi$ in $\Mod{\Tth}$ which is given by $\psi\whiteright \varphi = \psi\cdot e_1\blackright\varphi$.
\[
\xymatrix{X & Y\ar@{-_{>}}|-{\object@{o}}[l]_{\phi}\ar@{._{>}}|-{\object@{o}}[dl]^{\psi\whiteright \phi}\\
  G\ar@{-_{>}}|-{\object@{o}}[u]^\psi}
\]
\end{lemma}

Every $u\in\V$ can be interpreted as a $\Tth$-distributor $u:1\kmodto 1$, and then $u\otimes v$ corresponds to $v\kleisli u$. Liftings can be used to turn the ordered set $\Mod{\Tth}(G,X)$ into a $\V$-category by putting
\[
 [\varphi,\varphi']=\varphi\whiteright\varphi'=
\bigwedge_{x\in X}\hom(\varphi(x),\varphi'(x)),
\]
for all $\varphi,\varphi':G\kmodto X$. Hence, for a $\Tth$-category $X=(X,a)$, the $\V$-category $\Mod{\Tth}(G,X)$ is just the dual of the underlying $\V$-category $(VX)_0$ of $VX$. For every $\Tth$-distributor $\psi:X\kmodto Y$, composition with $\psi$ defines a mapping
\[
 \psi\kleisli-:\Mod{\Tth}(G,X)\to\Mod{\Tth}(G,Y)
\]
which is actually a $\V$-functor since
\[
 \varphi\whiteright\varphi'\le(\psi\kleisli\varphi)\whiteright(\psi\kleisli\varphi')
\]
follows from
\[
 \psi\kleisli\varphi\kleisli(\varphi\whiteright\varphi')\le \psi\kleisli\varphi'.
\]
One might hope that $\psi\kleisli-$ is even a $\Tth$-functor of type $VX\to VY$; unfortunately, this is in general not the case (see Proposition \ref{prop:nearlyOpen}). Fortunately, the situation is better if $\psi=f_\Tast$ for a $\Tth$-functor $f:X\to Y$, as we show next.
\begin{proposition}
Let $f:X\to Y$ be a $\Tth$-functor between $\Tth$-categories. Then $Vf:=f_\Tast\kleisli-:VX\to VY$ is a $\Tth$-functor. If, moreover, $X$ and $Y$ are representable and $f$ is a pseudo-homomorphism, then $Vf$ is a homomorphism.
\end{proposition}
\begin{proof}
Let $X=(X,a)$ and $Y=(Y,b)$ be $\Tth$-categories and $f:X\to Y$ be a $\Tth$-functor. We put $\Phi:=f_\Tast\kleisli-$, and show first that $b_0\cdot f\cdot\ev_X\le\ev_Y\cdot\Phi$,
\[
 \xymatrix{\V^X\ar[r]^\Phi\ar[d]|-{\object@{|}}_{\ev_X}\ar@{}[dr]|\le & \V^Y\ar[d]|-{\object@{|}}^{\ev_Y}\\ X\ar[r]|-{\object@{|}}_{b_0\cdot f} & Y}
\]
with equality if $X$ and $Y$ are representable and $f$ is a pseudo-homomorphism. In fact, for $\varphi\in\V^X$ and $y\in Y$,
\begin{align*}
 b_0\cdot f\cdot\ev_X(\varphi,y)=\bigvee_{x\in X}\varphi(x)\otimes b_0(f(x),y)
\intertext{and}
 \ev_Y\cdot\Phi(\varphi,y)=\bigvee_{\fx\in TX}\xi\cdot T\varphi(\fx)\otimes b(Tf(\fx),y).
\end{align*}
Restricting the second  formula to elements of the form $\fx=e_X(x)$ with $x\in X$ gives the first claim. Assume now that $X$ and $Y$ are representable with $a=a_0\cdot\alpha$ and $b=b_0\cdot\beta$ and that $f$ is a pseudo-homomorphism. Since $\varphi:X\to\V$ is a $\Tth$-functor, $\xi(T\varphi(\fx))\le\varphi(\alpha(\fx))$ for every $\fx\in TX$ (see Lemma \ref{lem:MapsReprCats}), and therefore
\[
 \bigvee_{\fx\in TX}\xi\cdot T\varphi(\fx)\otimes b(Tf(\fx),y)
\le \bigvee_{\fx\in TX}\varphi(\alpha(\fx))\otimes b_0(f(\alpha(\fx)),y)
=\bigvee_{x\in X}\varphi(x)\otimes b_0(f(x),y).
\]
Hence, for any $\fp\in T(\V^X)$,
\begin{multline*}
\Phi\cdot\mate{\mu_X}(\fp)
= f_\Tast\cdot\Txi a\cdot\Txi\Txi\ev_X\cdot Ti_\fp\cdot e_1
=f_\Tast\cdot\Txi a\cdot\Txi\Txi\ev_X\cdot e_{T(\V^X)}\cdot i_\fp\\
\le f_\Tast\cdot\Txi a\cdot\Txi\Txi\ev_X\cdot m_{\V^X}^\circ\cdot i_\fp
=f_\Tast\cdot\Txi a\cdot m_X^\circ\cdot\Txi\ev_X\cdot i_\fp=f_\Tast\cdot\Txi\ev_X\cdot i_\fp
\end{multline*}
and
\[
 \mate{\mu_Y}\cdot T\Phi(\fp)=b\cdot\Txi\ev_Y\cdot T\Phi(\fp,-)
\ge b\cdot\Txi b_0\cdot Tf\cdot\Txi\ev_X(\fp,-)
=f_\Tast\cdot\Txi\ev_X(\fp,-);
\]
and we conclude that $\Phi$ is indeed a $\Tth$-functor of type $(\V^X)^\op\to(\V^Y)^\op$. If $f:X\to Y$ is a pseudo-homomorphism between representable $\Tth$-categories, then the second inequality is actually an equality thanks to the calculations above. Finally, 
\[
 f_\Tast\cdot\Txi a\cdot\Txi\Txi\ev_X\cdot Ti_\fp\cdot e_1
\ge f_\Tast\cdot e_X\cdot a\cdot\Txi\ev_X\cdot i_\fp
\]
and
\[
 f_\Tast\cdot e_X\cdot a=b_0\cdot f\cdot a\le b_0\cdot f\cdot\alpha=b_0\cdot b\cdot Tf=f_\Tast;
\]
and this shows that also the first inequality becomes also an equality in this case.
\end{proof}

Therefore $V$ can be seen as an endofunctor
\[
 V:\Cat{\Tth}\to\Cat{\Tth}
\]
on $\Cat{\Tth}$, and also as an endofunctor
\[
 V:\Repr{\Tth}\to\Repr{\Tth}
\]
on $\Repr{\Tth}$. Furthermore, in both cases $V$ is actually a 2-functor since, for all $\Tth$-functors $f,g:X\to Y$, $f\le g$ is equivalent to $f_\Tast\ge g_\Tast$, and therefore implies $f_\Tast\kleisli-\ge g_\Tast\kleisli-$ which is equivalent to $Vf\le Vg$.

\begin{remark}\label{rem:fVastfast}
If $f:(X,a)\to(Y,b)$ is a pseudo-homomorphism (where $a=a_0\cdot\alpha$) and $\varphi:Z\kmodto X$, then
\[
 f_\Tast\kleisli\varphi
=b\cdot Tf\cdot\Txi\varphi\cdot m_Z^\circ
=b_0\cdot f\cdot\alpha\cdot\Txi\varphi\cdot m_Z^\circ
=b_0\cdot f\cdot\varphi
=f_*\cdot\varphi.
\]
\end{remark}

Comparing with the situation for $\V$-categories, one might expect $Vf:VX\to VY$ to be right adjoint. If it is so, its left adjoint is necessarily given by $f^\Tast\kleisli-$, which is the dual of the exponential $\V^f:\V^Y\to\V^X$, for $f:X\to Y$. Therefore we have to investigate whether or not $\V^f:\V^Y\to\V^X$ is a homomorphism (see Lemma \ref{lem:dualTVGraphFun}); as it turns out, this is only true under additional conditions on $f$.

\begin{definition}\label{def:downwards_open}
A $\Tth$-functor $f:(X,a)\to(Y,b)$ is called \emph{downwards open} whenever $a\cdot Tf^\circ\cdot\Txi b_0\ge f^\circ\cdot b$.
\end{definition}
Note that every $\Tth$-functor $f:(X,a)\to(Y,b)$ satisfies $a\cdot Tf^\circ\cdot \Txi b_0\le f^\circ\cdot b$, hence for a downwards open $\Tth$-functor one actually has equality. Of course, every open $\Tth$-functor (see \citep[Definition 4.1]{CH04a}) is downwards open, and the reverse is true whenever $b_0=1_Y$. In particular, any non-open continuous map between compact Hausdorff spaces is an example of a homomorphism of representable spaces which is not downwards open.
\begin{lemma}
For every $\V$-functor $f:(X,c)\to(Y,d)$, the $\Tth$-functor $f:(X,e_X^\circ\cdot\Txi c)\to(Y,e_Y^\circ\cdot\Txi d)$ is downwards open. In particular, every $\V$-functor is downwards open in $\Cat{\V}$.
\end{lemma}
\begin{proof}
We put $a=e_X^\circ\cdot\Txi c$ and $b=e_Y^\circ\cdot\Txi d$. Then
$
 f^\circ\cdot b=e_X^\circ\cdot Tf^\circ\cdot \Txi d\le e_X^\circ\cdot Tf^\circ\cdot\Txi b_0
$.
\end{proof}

\begin{lemma}
For every $\Tth$-category $X=(X,a)$, $e_X:X\to TX$ is downwards open if and only if $X$ is core-compact.
\end{lemma}
\begin{proof}
Recall that the $\Tth$-category structure on $TX$ is given by $\Txi a\cdot m_X^\circ\cdot m_X$, and the underlying $\V$-category structure is $\Txi a\cdot m_X^\circ$. We compute
\[e_X^\circ\cdot\Txi a\cdot m_X^\circ\cdot m_X=(e_X\kleisli a)\cdot m_X=a\cdot m_X\]
and
\[
a\cdot Te_X^\circ\cdot\Txi(\Txi a\cdot m_X^\circ) =a\cdot\Txi(e_X\kleisli a)=a\cdot\Txi a,
\]
which proves the claim.
\end{proof}

\begin{lemma}\label{lem:FFDO_CoreComp}
Let $f:X\to Y$ be a fully faithful downwards open $\Tth$-functor between $\Tth$-categories where $Y$ is core-compact. Then $X$ is core-compact.
\end{lemma}
\begin{proof}
Let $X=(X,a)$ and $Y=(Y,b)$ be $\Tth$-categories with $b\cdot\Txi b=b\cdot m_Y$, and let $f:X\to Y$ be a fully faithful downwards open $\Tth$-functor. Then
\begin{multline*}
a\cdot\Txi a=f^\circ\cdot b\cdot Tf\cdot Tf^\circ\cdot\Txi b\cdot TTf
\ge f^\circ\cdot f\cdot a\cdot Tf^\circ\cdot\Txi b_0\cdot\Txi b\cdot TTf\\
=f^\circ\cdot b\cdot \Txi b\cdot TTf
=f^\circ\cdot b\cdot m_Y\cdot TTf=f^\circ\cdot b\cdot Tf\cdot m_X=a\cdot m_X. \qedhere
\end{multline*}
\end{proof}

\begin{example}
A continuous map $f:X\to Y$ between topological spaces is downwards open if and only if the down closure $\downc f(A)$ of every open subset $A\subseteq X$ is open in $Y$. To see this, we recall (Example \ref{ex:OrdMetComp}) that
\[
 \fx\,(\Uxi\!\le)\,\fx'\iff \forall A\in\fx'\,.\,\downc A\in\fx,
\]
where $\le$ is an order relation on $X$ and $\fx,\fx'\in UX$. Assume first that $f$ is downwards open and let $A\subseteq X$ open. Let $\fy\to y\le f(x)$ in $Y$, for some $x\in A$. By hypothesis, there is some $\fx\in UX$ with $\fx\to x$ and $\fy\,(\Uxi\!\le)\,Uf(\fx)$. From $A$ open it follows that $A\in\fx$, hence $\downc f(A)\in\fy$. Conversely, assume now that $\downc f(A)$ is open, for every open subset $A\subseteq X$. Let $x\in X$ and $\fy\in UY$ with $\fy\to f(x)$. Then the ideal
\[
 \{A\subseteq X\mid \downc f(A)\notin\fy\}
\]
is disjoint from the neighbourhood filter of $x$, and therefore (see Proposition \ref{prop:UltExtExc}) there is some ultrafilter $\fx\in UX$ with $\fx\to x$ and $\fy\,(\Uxi\!\le)\,Uf(\fx)$.
\end{example}

\begin{proposition}\label{prop:nearlyOpen}
Let $f:(X,a)\to(Y,b)$ be a $\Tth$-functor between $\Tth$-categories. Then the following assertions are equivalent.
\begin{eqcond}
\item\label{cond:nearlyOpen1} $f$ is downwards open.
\item\label{cond:nearlyOpen2} $\V^f:\V^Y\to\V^X$ is a homomorphism.
\item\label{cond:nearlyOpen3} $Vf:VX\to VY$ has a left adjoint.
\end{eqcond}
\end{proposition}
\begin{proof}
Assume first that $f$ is downwards open. Let $\fq\in T(\V^Y)$ and $x\in X$. Then
\begin{align*}
 \V^f(\mate{\mu_Y}(\fq))(x)&=\mate{\mu_Y}(\fq)(f(x))=b\cdot \Txi\ev_Y(\fq,f(x))=f^\circ\cdot b\cdot\Txi\ev_Y(\fq,x)
\intertext{and}
 \mate{\mu_X}(T(\V^f)(\fq))(x)&=a\cdot\Txi\ev_X(T(\V^f)(\fq),x).
\end{align*}
Furthermore, the diagram
\[
 \xymatrix{\V^Y\ar[r]^{\V^f}\ar[d]|-{\object@{|}}_{\ev_Y} & \V^X\ar[d]|-{\object@{|}}^{\ev_X}\\ Y\ar[r]|-{\object@{|}}_{f^\circ\cdot b_0} & X}
\]
commutes since, for every $\varphi:\V^Y$ and $x\in X$,
\[
\ev_X\cdot\V^f(\varphi,x)=\varphi(f(x))=b_0\cdot\varphi(f(x))=f^\circ\cdot b_0\cdot\ev_Y(\varphi,x).
\]
Therefore $a\cdot\Txi\ev_X(T(\V^f)(\fq),x)=a\cdot Tf^\circ\cdot\Txi b_0\cdot \Txi\ev_Y(\fq,x)$. We conclude that $\V^f$ is a homomorphism, hence \eqref{cond:nearlyOpen1}$\Rw$\eqref{cond:nearlyOpen2}. Since $\ev_Y\cdot\coyoneda_Y=b_0$, from $f^\circ\cdot b\cdot \Txi\ev_Y=a\cdot Tf^\circ\cdot\Txi b_0\cdot\Txi\ev_Y$ one obtains $f^\circ\cdot b=a\cdot Tf^\circ\cdot\Txi b_0$; hence \eqref{cond:nearlyOpen2}$\Rw$\eqref{cond:nearlyOpen1}. Finally, the equivalence \eqref{cond:nearlyOpen2}$\RLw$\eqref{cond:nearlyOpen3} is clear.
\end{proof}

Recall that the structure $a$ of a $\Tth$-category $X=(X,a)$ can be seen as a $\Tth$-functor $a:(TX)^\op\otimes X\to\V$, and therefore induces a morphism
\[
\mate{a}:(TX)^\op\to\V^X,\,\fx\mapsto a(\fx,-)
\]
in $\Gph{\Tth}$. 

\begin{proposition}\label{prop:coYoneda}
For each $\Tth$-category $X=(X,a)$, the \emph{Yoneda map} \[\coyoneda_X:X\to VX,\,x\mapsto a_0(x,-)\] is a fully faithful and downwards open $\Tth$-functor. If, moreover, $X$ is representable, then $\coyoneda_X:X\to VX$ is even a pseudo-homomorphism. Furthermore, $\coyoneda_X^\Tast(\fp,x)=\mu(\fp,x)$ for all $x\in X$ and $\fp\in T(\V^X)$.
\end{proposition}
\begin{proof}
Let $X=(X,a)$ be a $\Tth$-category. Since the diagram
\[
 \xymatrix{X\times X\ar[r]^{\coyoneda_X\times 1_X}\ar[rd]_{a_0} & \V^X\times X\ar[d]^\ev\\ & \V}
\]
commutes, for $\fp=T\coyoneda_X(\fx)$ with $\fx\in TX$ one has (where $x\in X$)
\[
\mu(\fp,x)=a\cdot\Txi\ev(T\coyoneda_X(\fx),x)=a\cdot\Txi a_0(\fx,x)=a(\fx,x)
\]
and therefore $\fspstrV{T\coyoneda_X(\fx)}{\coyoneda_X(x)}=a(\fx,x)$. We conclude that $\coyoneda_X:X\to VX$ is a fully faithful $\Tth$-functor. If $X$ is representable, just note that
\[
 \mu\cdot T\coyoneda_X(\fx)=a(\fx,-)=a_0(\alpha(\fx),-)=\coyoneda\cdot\alpha(\fx),
\]
where $\alpha:TX\to X$ is the pseudo-algebra structure of $X$. Furthermore,
\[
 \coyoneda_X^\Tast(\fp,x)=\fspstrV{\fp}{\coyoneda_X(x)}=[\coyoneda_X(x),\mate{\mu}(\fp)]=\mu(\fp,x),
\]
for all $x\in X$ and $\fp\in T(\V^X)$; where the last equality follows from the Yoneda Lemma for $\V$-categories. From that it follows that
\begin{align*}
 \coyoneda_X^\circ\cdot\fspstrV{-}{-}=\mu &&\text{and}&&
 \coyoneda_X^\circ\cdot\langle-,-\rangle=\ev,
\end{align*}
and from the latter equation we deduce $a\cdot T\yoneda_X^\circ\cdot\Txi\langle-,-\rangle=a\cdot\Txi\ev=\mu$. Therefore $\coyoneda_X:X\to VX$ is downwards open.
\end{proof}

\begin{corollary}
Let $X$ be a $\Tth$-category. Then the following assertions are equivalent.
\begin{eqcond}
\item\label{cond:CoreComp1}  $X$ is core-compact.
\item $VX$ is representable.
\item\label{cond:CoreComp3} $VX$ is core-compact.
\item\label{cond:CoreComp4} The $\Tth$-graph $\V^X$ is a $\Tth$-category.
\end{eqcond}
\end{corollary}
\begin{proof}
If $X$ is core-compact, then $VX=(\V^X)^\op$ is representable by Corollary \ref{cor:reprspaceVpowerX}, hence core-compact by Proposition \ref{prop:CharReprCat}. The implication \eqref{cond:CoreComp3}$\Rw$\eqref{cond:CoreComp1} follows from Lemma \ref{lem:FFDO_CoreComp} and Proposition \ref{prop:coYoneda}, and the equivalence \eqref{cond:CoreComp3}$\RLw$\eqref{cond:CoreComp4} from Proposition \ref{prop:DualVsCore}.
\end{proof}

\begin{proposition}
Let $X=(X,a)$ be a $\Tth$-category. Then $X$ is core-compact if and only if $\mate{a}:(TX)^\op\to\V^X$ is a homomorphism. In this case, $\coyoneda_X={\mate{a}}^\op\cdot e_X$.
\end{proposition}
\begin{proof}
Just note that, for every $\fX\in TTX$,
\begin{align*}
 \mate{a}\cdot m_X(\fX) &= a(m_X(\fX),-)
\intertext{and}
 \mu\cdot T\mate{a}(\fX) &=a\cdot\Txi\ev(T\mate{a}_X(\fX),-)=a\cdot\Txi a(\fX,-),
\end{align*}
where the last equality follows from the fact that
\[
 \xymatrix{TX\times X\ar[r]^{\mate{a}\times 1_X}\ar[rd]_a & \V^X\times X\ar[d]^\ev\\ & \V}
\]
commutes.
\end{proof}

From $f_\Tast\kleisli x_\Tast=f(x)_\Tast$ it follows that $\coyoneda=(\coyoneda_X)$ is a natural transformation $\coyoneda:1\to V$, for $V:\Cat{\Tth}\to\Cat{\Tth}$ and for $V:\Repr{\V}\to\Repr{\V}$. We will now show that $V$ is part of a monad on both $\Repr{\Tth}$ and $\Cat{\Tth}$. 

Recall from Proposition \ref{prop:coYoneda} that $\coyoneda_X:X\to VX$ is downwards open, hence, by Proposition \ref{prop:nearlyOpen}, $V\coyoneda_X:VX\to VVX$ has a left adjoint given by
\[
 \coyonmult_X:=\coyoneda_X^\Tast\kleisli-:VVX\to VX.
\]
We show now that $\coyonmult=(\coyonmult_X)$ is the multiplication of a monad $\V=\vmonad$.

\begin{lemma}
For every $\Tth$-category $X$, $\coyonmult_X:VVX\to VX$ is right adjoint to $\coyoneda_{VX}:VX\to VVX$.
\end{lemma}
\begin{proof}
Let first $\varphi\in\V^X$. Then
\[
 \coyoneda_X^\Tast\kleisli\varphi_\Tast=[\coyoneda_X(-),\varphi]=\varphi,
\]
hence $\coyonmult_X\cdot\coyoneda_{VX}=(\coyoneda_X^\Tast\kleisli-)\cdot \coyoneda_{VX}=1$. Let now $\Phi\in\V^{(VX)}$. For every $\varphi\in\V^X$ and $x\in X$,
\[
 [x_\Tast,\varphi]\otimes\Phi(\varphi)\le\Phi(x_\Tast)
\]
since $\Phi:VX\to\V$ is a $\Tth$-functor, hence
\[
 \Phi(\varphi)\le\bigwedge_{x\in X}\hom(\varphi(x),\Phi(x_\Tast))
\]
for all $\varphi\in\V^X$, that is $\Phi\le \coyoneda_{VX}\cdot(\coyoneda_X^\Tast\kleisli-)(\Phi)$. Therefore $\coyoneda_{VX}\cdot\coyonmult_X\le 1_{VX}$.
\end{proof}

By the lemma above,
\[
 \coyoneda_{VX}\dashv\coyonmult_X\dashv V\coyoneda_X
\]
in $\Cat{\Tth}$ and, if $X$ is representable, even in $\Repr{\Tth}$. Furthermore, $\coyonmult_X\cdot \coyoneda_{VX}=1_{VX}=\coyonmult_X\cdot V\coyoneda_X$ since $\coyoneda_{VX}$ is fully faithful. For every $f:X\to Y$ in $\Cat{\Tth}$, naturality of $\coyoneda$ gives us
\begin{align*}
 Vf=\coyonmult_Y\cdot VVf\cdot\yoneda_{VX} &&\text{and}&& Vf=\coyonmult_Y\cdot VVf\cdot V\yoneda_{X},
\end{align*}
hence $Vf\cdot\coyonmult_X\le\coyonmult_Y\cdot VVf$ and $Vf\cdot\coyonmult_X\ge\coyonmult_Y\cdot VVf$ since $\coyoneda_{VX}\dashv\coyonmult_X\dashv V\coyoneda_X$. Consequently, $\coyonmult=(\coyonmult_X):VV\to V$ is a natural transformation for $V:\Cat{\Tth}\to\Cat{\Tth}$ and, since every $\coyonmult_X$ is left adjoint, also for $V:\Repr{\Tth}\to\Repr{\Tth}$ (see Proposition \ref{prop:TotCoComplRepr}). Finally, for every $\Tth$-category $X$, the diagram
\[
 \xymatrix{VVVX\ar[d]_{\coyonmult_{VX}}\ar[r]^{V\!\coyonmult_X} & VVX\ar[d]^{\coyonmult_X}\\ VVX\ar[r]_{\coyonmult_X} & VX}
\]
commutes since the diagram of the corresponding left adjoints does. All told:
\begin{theorem}\label{thm:Vietoris}
$\mV=\vmonad$ is a monad on $\Cat{\Tth}$ and on $\Repr{\Tth}$ which, moreover, is of Kock-Z\"oberlein type.
\end{theorem}

\begin{remark}\label{rem:MonTildeV}
The monad $\mV=\vmonad$ on $\Repr{\Tth}$ restricts to a monad $\mV=\vmonad$ on $\Repr{\Tth}_\sep$ since $VX$ is separated, for every $\Tth$-category $X$; and the categories $(\Repr{\Tth})^\mV$ and $(\Repr{\Tth}_\sep)^\mV$ of Eilenberg--Moore algebras are actually equal. Furthermore, we can lift $\mV$ to a monad $\mtV=\tvmonad$ on $\SET^\mT$ via the adjunction (see Remark \ref{rem:AdjReprTalg})
\[
 \Repr{\Tth}_\sep\adjunct{}{}\SET^\mT,
\]
that is, $\mtV=\tvmonad$ is the monad on $\SET^\mT$ induced by the composite of the adjunctions
\[
 (\Repr{\Tth}_\sep)^\mV\adjunct{}{}\Repr{\Tth}_\sep\adjunct{}{}\SET^\mT.
\]
Explicitly, the unit $\widetilde{\coyoneda}_X:X\to\widetilde{V}X$ at $X$ is defined by $i\cdot\widetilde{\coyoneda}_X=\coyoneda_X$, where $i:\widetilde{V}X\to VX,\,\varphi\mapsto\varphi$; and the multiplication $\widetilde{\coyonmult}_X$ at $X$ sends $\Phi:G\modto\widetilde{V}X$ to $\coyoneda_X^\Tast\kleisli i_\Tast\kleisli\Phi$. We shall see in Remark \ref{rem:TildeVviaRepr} that $\mtV=\tvmonad$ is also induced by the adjunction
\[
 \Cocts{\Tth}_\sep\adjunct{}{}\SET^\mT
\]
of Remark \ref{rem:CoctsOverSetT}.
\end{remark}

\begin{example}\label{ex:tildeVApp}
In the topological case, the monad $\mtV=\tvmonad$ on $\COMPHAUS$ is the Vietoris monad whose functor part was originally studied in \cite{Vie22}. In the approach case (i.e.\ $\Tth=\Uth_{\Pp}$), we obtain a monad $\mtV=\tvmonad$ on $\COMPHAUS$ where $\widetilde{V}X$ is the compact Hausdorff space with underlying set all lower semi-continuous functions of type $X\to[0,\infty]$ (see Example \ref{ex:topappV}), and where an ultrafilter $\fp$ on $\widetilde{V}X$ converges to (see Corollary \ref{cor:reprspaceVpowerX})
\[
 X\to[0,\infty],\; x\mapsto \inf\left\{\sup_{\calA\in\fp,A\in\fx}\inf_{\varphi\in\calA,z\in A}\varphi(z)\mid \fx\in UX, \fx\to x\right\}.
\]
The monad $\mtV=\tvmonad$ on $\COMPHAUS$ obtained from $\Uth_{\Pm}$ has the same functor and the same unit as for $\Uth_{\Pp}$, but the multiplication is different.
\end{example}

\section{Complete $\Tth$-categories}\label{sect:Cts}

Similar to what was done for colimits, we introduce now a notion of weighted limit following closely the $\V$-categorical case. A \emph{weighted limit diagram} $(h,\varphi)$ in a $\Tth$-category $X$ is given by a $\Tth$-functor $h:A\to X$ and a $\Tth$-distributor $\varphi:G\kmodto A$,
\[
 \xymatrix{A\ar[r]^h & X\\ G\ar@{-_{>}}|-{\object@{o}}[u]^\varphi}
\]
and $x_0\in X$ is a \emph{limit} of this diagram, written as $x_0\simeq\lim(h,\varphi)$, if $x_0$ represents $\varphi\whiteright h^\Tast$ in the sense that $x_0^\Tast=\varphi\whiteright h^\Tast$. We hasten to remark that \emph{we cannot consider an arbitrary $\Tth$-distributor $\varphi:B\kmodto A$ above since the lifting $\varphi\whiteright h^\Tast$ might not exist}. A $\Tth$-functor $f:X\to Y$ \emph{preserves the limit} of $h$ and $\varphi$ whenever $f(x_0)\simeq\lim(f\cdot h,\varphi)$, and $f:X\to Y$ is said to be \emph{continuous} whenever $f$ preserves all weighted limits which exist in $X$. Note that, for any $\fx\in TX$,
\[
 \varphi\whiteright h^\Tast(\fx)=\bigwedge_{z\in A}\hom(\varphi(z),a(\fx,h(z))),
\]
hence $x_0\simeq\lim(h,\varphi)$ precisely when, for all $\fx\in TX$,
\[
 a(\fx,x_0)=\bigwedge_{z\in A}\hom(\varphi(z),a(\fx,h(z))).
\]
In particular, the equality above holds for all $\fx=e_X(x)$, $x\in X$, therefore $x_0$ is also a limit of the underlying diagram in $\Cat{\V}$; however, a $\V$-categorical limit in $X_0$ does not need to be a limit in the $\Tth$-category $X$ (see Example \ref{ex:NotTopLim}). Nevertheless, if we know that a diagram has a limit in $X$, then this limit can be calculated in the underlying $\V$-category $X_0$.

\begin{remark}
A particular instance of a weighted limit in a topological space was considered in \citep{Luc11} and called directed conjunction there.
\end{remark}

\begin{example}\label{ex:NotTopLim}
We consider the ordered set $X_0=[0,1]$ (closed unit interval) with the usual order $\le$, and let $X$ be the induced Alexandroff space. Hence, the closed subsets of $X$ are precisely the up-closed subsets of $[0,1]$, and an ultrafilter $\fx\in UX$ converges to $x\in X$ if and only if, for all $A\in\fx$, there exists some $y\in A$ with $y\le x$. Clearly, $0$ is the infimum of $A=]0,1]$ in $X_0$, but we shall see that $0$ is not an infimum of the closed subset $A=]0,1]$ of $X$ in $\TOP$. In fact, let $\fx$ be any ultrafilter containing the sets $]0,r]$, for $r>0$. By construction, $\fx$ converges to every $x\in X$ different from $0$, but not to $x=0$.
\end{example}

A $\Tth$-category $X$ is called \emph{complete} whenever every weighted limit diagram in $X$ has a limit. To check for completeness it is enough to consider weighted limit diagrams where $h$ is the identity $1_X:X\to X$ on $X$ since a limit of $(h,\varphi)$ is a limit of $(1_X,h^\Tast\kleisli\varphi)$, and vice versa. A limit of $(1_X,\varphi)$ we also call \emph{infimum} of $\varphi$ and we denote such a limit as $\Inf_X(\varphi)$ (or simply $\Inf(\varphi)$). Note that, for a $\Tth$-category $X=(X,a)$ and $\varphi:G\kmodto X$, $x_0\simeq\Inf_X(\varphi)$ precisely when, for all $\fx\in TX$,
\[
 a(\fx,x_0)=\bigwedge_{x\in X}\hom(\varphi(x),a(\fx,x)).
\]
If $f:X\to Y$ is a $\Tth$-functor between complete $\Tth$-categories, then $f$ is continuous if and only if $f$ sends, for any $\varphi:G\kmodto X$, a infimum of $\varphi$ to a limit of the diagram $(f,\varphi)$ in $Y$. Since both limits can be calculated in the underlying $\V$-categories, we find that the $\Tth$-functor $f$ is continuous if the underlying $\V$-functor is continuous. In fact:

\begin{lemma}\label{lem:TimpliesVcompl}
For every complete $\Tth$-category $X$, the $\V$-category $X_0$ is complete. Moreover, a $\Tth$-functor $f:X\to Y$ between complete $\Tth$-categories is continuous if and only if the underlying $\V$-functor $f:X_0\to Y_0$ is continuous.
\end{lemma}
\begin{proof}
Every $\V$-distributor $\varphi:G\modto X_0$ can be seen as a $\Tth$-distributor $\varphi:G\kmodto A(X_0)$, and a limit of the diagram given by $\varphi:G\kmodto A(X_0)$, $A(X_0)\to X$ is also a infimum of $\varphi:G\modto X_0$ in $X_0$. The second claim is now clear.
\end{proof}

Similarly as in \eqref{eq:ext_adj} in Subsection \Roman{subsect:one} of Section \ref{sect:setting}, one has

\begin{lemma}\label{lem:CompExtDist}
Let $\varphi:G\kmodto X$, $\beta:Y\kmodto X$ and $\alpha:Z\kmodto X$ be $\Tth$-distributors where $\alpha$ is left adjoint. Then
\[
\varphi\whiteright(\beta\kleisli\alpha)=(\varphi\whiteright\beta)\kleisli\alpha.
\]
\end{lemma}
As expected:
\begin{proposition}\label{prop:RighAdjCont}
Every right adjoint $\Tth$-functor is continuous.
\end{proposition}
\begin{proof}
For a right adjoint $\Tth$-functor $f:X\to Y$, a $\Tth$-distributor $\varphi:G\kmodto A$, a $\Tth$-functor $h:A\to X$ and $x_0\in X$ with $x_0\simeq\lim(h,\varphi)$, we calculate
\[
\varphi\whiteright(f\cdot h)^\Tast
=(\varphi\whiteright h^\Tast)\kleisli f^\Tast
=x_0^\Tast\kleisli f^\Tast=f(x_0)^\Tast.
\]
Here we use Lemma \ref{lem:CompExtDist} and that $f^\Tast$ is a left adjoint $\Tth$-distributor since $f$ is right adjoint.
\end{proof}

The following proposition is in sharp contrast to the case of weighted colimits where the existence of all colimits with weights $X\kmodto G$ does not guarantee the existence of a left adjoint of the Yoneda embedding $\yoneda_X:X\to PX$ in $\Cat{\Tth}$.

\begin{proposition}
Let $X=(X,a)$ be a $\Tth$-category. Then $X$ is complete if and only if $\coyoneda_X:X\to VX$ has a right adjoint $\Inf_X:VX\to X$ in $\Cat{\Tth}$. 
\end{proposition}
\begin{proof}
First recall that $\mu\cdot T\coyoneda_X=a\cdot\Txi\ev\cdot T\coyoneda_X=a\cdot\Txi a_0=a$, hence $a(\fx,-)=\mate{\mu}(T\coyoneda_X(\fx))$ for all $\fx\in TX$. Therefore $X$ is complete if and only if there exists a map $\Inf_X:VX\to X$ satisfying
\[
 a(\fx,\Inf_X(\varphi))=[\varphi,a(\fx,-)]
=[\varphi,\mu_X(T\coyoneda_X(\fx))]
=\fspstrV{T\coyoneda_X(\fx)}{\varphi},
\]
for all $\fx\in TX$ and $\varphi:G\kmodto X$. But this conditions just means that $(\coyoneda_X)_\Tast=\Inf_X^\Tast$, and the assertion follows using \eqref{eq:DistVsFun} of Subsection \Roman{subsect:seven} of Section \ref{sect:setting}. 
\end{proof}

\begin{proposition}
Let $X=(X,a)$ and $Y=(Y,b)$ be complete $\Tth$-categories. Then a $\Tth$-functor $f:X\to Y$ is continuous if and only if the diagram
\[
\xymatrix{VX\ar[d]_{\Inf_X}\ar[r]^{Vf}\ar@{}[dr]|{\simeq} & VY\ar[d]^{\Inf_Y}\\ X\ar[r]_f & Y}
\]
commutes up to equivalence.
\end{proposition}

\begin{theorem}
The category $(\Cat{\Tth})^\mV$ of Eilenberg--Moore algebras of $\mV$ is precisely the category of complete and separated $\Tth$-categories and continuous $\Tth$-functors.
\end{theorem}

\begin{proposition}
Let $X$ be a core-compact $\Tth$-category. Then $X$ is complete if and only if $X$ is representable and $X_0$ is a complete $\V$-category.
\end{proposition}
\begin{proof}
Let $X=(X,a)$ be a core-compact $\Tth$-category. Assume first that $X$ is complete. Then $X_0$ is complete by Lemma \ref{lem:TimpliesVcompl}, and a left inverse of $e_X:X\to TX$ is given by $\Inf_X\cdot{\mate{a}}^\op$. Assume now that $X$ is representable (with algebra structure $\alpha:TX\to X$) and that $X_0$ is complete. Let $\varphi:G\kmodto X$ be a $\Tth$-distributor, then $\varphi$ can be also seen as a $\V$-distributor $\varphi:G\modto X_0$. Let $x$ be an infimum of $\varphi:G\modto X_0$ in $X_0$. Then, for every $\fx\in TX$,
\[
 a(\fx,x)=a_0(\alpha(\fx),x)
=\bigwedge_{x\in X}\hom(\varphi(x),a_0(\alpha(\fx),x))
=\bigwedge_{x\in X}\hom(\varphi(x),a(\fx,x)),
\]
hence $x$ is also an infimum of $\varphi:1\kmodto X$.
\end{proof}

\begin{proposition}
Let $X$ and $Y$ be $\Tth$-categories where $X$ is complete. Then the following assertions are equivalent for a $\Tth$-functor $f:X\to Y$.
\begin{eqcond}
\item\label{condAd1} $f$ is right adjoint.
\item\label{condAd2} $f$ preserves limits and $Vf$ is right adjoint.
\item\label{condAd3} $f$ preserves limits and is downwards open.
\end{eqcond}
\end{proposition}
\begin{proof}
The equivalence \eqref{condAd2}$\RLw$\eqref{condAd3} follows from Proposition \ref{prop:nearlyOpen}. The assertion \eqref{condAd2} follows from \eqref{condAd1} by Proposition \ref{prop:RighAdjCont} and the fact that $V$ is a 2-functor. Finally, assuming \eqref{condAd2}, a left adjoint of $f$ is given by $\Inf_X\cdot G\cdot\coyoneda_Y$ where $G$ is a left adjoint of $Vf$.
\end{proof}

\begin{corollary}
Let $X$ be a $\Tth$-category. Then $X$ is complete if and only if $X$ is injective in $\Cat{\Tth}$ with respect to fully faithful downwards open $\Tth$-functors.
\end{corollary}
\begin{proof}
This is an instance of \citep[Theorem 4.2.2]{Esc98}.
\end{proof}

\section{Isbell conjugation adjunction}\label{sect:Isbell}

For every $\Tth$-category $X=(X,a)$, there is an adjunction
\[
 \Mod{\Tth}(G,X)^\op\adjunct{(-)^+}{(-)^-}\Mod{\Tth}(X,G)
\]
in $\Cat{\V}$ where
\begin{align*}
\varphi^-(\fx)&=\varphi\whiteright 1_X^\Tast(\fx)
=\bigwedge_{x\in X}\hom(\varphi(x),a(\fx,x))
\intertext{for all $\varphi:G\kmodto X$ and $\fx\in TX$, and}
\psi^+(x)&=1_X^\Tast\whiteleft\psi(x)
=\bigwedge_{\fx\in TX}\hom(\psi(\fx),a(\fx,x))
\end{align*}
for all $\psi:X\kmodto G$ and $x\in X$. Following \citep{Woo04}, we refer to this adjunction as an \emph{Isbell conjugation adjunction}. Note that
\begin{align*}
(x_\Tast)^-=x^\Tast &&\text{and}&& (x^\Tast)^+=x_\Tast,
\end{align*}
for all $x\in X$. Furthermore $x\in X$ is an infimum of $\varphi:G\kmodto X$ if and only if $x^\Tast=\varphi^-$, and $x$ is a supremum of $\psi:X\kmodto G$ if and only if $x_\Tast=\psi^+$.
\begin{proposition}
For every $\Tth$-distributor $\varphi:G\kmodto X$ and $x\in X$, $x$ is an infimum of $\varphi$ in $X$ if and only if $x$ is a supremum of $\varphi^-$ in $X$. Similarly, for every $\Tth$-distributor $\psi:X\kmodto G$ and $x\in X$, $x$ is a supremum of $\psi$ if and only if $x$ is an infimum of $\psi^+$.
\end{proposition}
\begin{proof}
If $x^\Tast=\varphi^-$, then $x_\Tast=(\varphi^-)^+$, and this in turn implies $x^\Tast=((\varphi^-)^+)^-=\varphi^-$. A similar argumentation proves the second claim.
\end{proof}
\begin{theorem}\label{thm:ComplIsCocompl}
A $\Tth$-category $X$ is complete if and only if $X$ is cocomplete.
\end{theorem}

\begin{lemma}
Let $f:(X,a)\to(Y,b)$ be a $\Tth$-functor where $(Y,b)$ is representable with $\mT$-algebra structure $\beta:TY\to Y$. Then
\[
 f_\Tast=\overline{f}_*,
\]
where $\overline{f}:TX\to Y$ is the extension of $f:(X,a)\to(Y,b)$ along $e_X:X\to TX$.
\end{lemma}
\begin{proof}
We calculate:
\[
 f_\Tast=b\cdot Tf=b_0\cdot\beta\cdot Tf=b_0\cdot\overline{f}=\overline{f}_*.\qedhere
\]
\end{proof}
\begin{proposition}
For every $\Tth$-category $X=(X,a)$, the $\V$-functor $(-)^-$ is actually a $\Tth$-functor $(-)^-:VX\to PX$.
\end{proposition}
\begin{proof}
Note that
\[
 (\coyoneda_X)_\Tast(\fx,\varphi)=\fspstrV{T\coyoneda_X(\fx)}{\varphi}
=[\varphi,\mu(T\coyoneda_X(\fx),-)]=[\varphi,a(\fx,-)]=\bigwedge_{x\in X}\hom(\varphi(x),a(\fx,x)),
\]
for all $\varphi\in VX$ and $\fx\in TX$; hence $(-)^-$ is the mate of $(\coyoneda_X)_\Tast:X\kmodto VX$.
\end{proof}
However, the $\V$-functor $(-)^+$ is in general not a $\Tth$-functor of type $PX\to VX$. For instance, the representable approach space $\Pp^\op$ is complete (and cocomplete) but not totally cocomplete, and therefore $(-)^+$ cannot be a $\Tth$-functor for $X=\Pp^\op$.

\section{Totally complete $\Tth$-categories}\label{sect:TotCts}

At the beginning of Section \ref{sect:Cts} we pointed already out that the notion of complete $\Tth$-category cannot be strengthened to ``totally complete'' exactly the same way as it was done for cocompleteness, namely by allowing all $\Tth$-distributors $\psi:B\kmodto A$ as limit weights. Nevertheless, in this section we introduce a notion of totall completeness which turns out to be the dual of total cocompleteness.

\begin{definition}
A representable $\Tth$-category $X=(X,a)$ is called \emph{totally complete} if $\coyoneda_X:X\to VX$ has a right adjoint $\Inf_X:VX\to X$ in $\Repr{\Tth}$.
\end{definition}

Hence, a totally complete $\Tth$-category $X$ is a complete representable $\Tth$-category where, moreover, $\Inf_X:VX\to X$ is a pseudo-homomorphism. We write
\[
 \Cts{\Tth}
\]
for the category of totally complete $\Tth$-categories and pseudo-homomorphisms which preserve limits, and $\Cts{\Tth}_\sep$ denotes its full subcategory defined by the separated $\Tth$-categories. By definition, $\Cts{\Tth}_\sep\simeq(\Repr{\Tth})^\mV$. Clearly, every $\Tth$-category of the form $VX$ is totally complete. Since $(PX)^\op=V((TX)^\op)$, this includes the duals of $\Tth$-categories of the form $PX$. In fact, we will show that the totally complete $\Tth$-categories are precisely the duals of totally cocomplete $\Tth$-categories.

\begin{lemma}\label{lem:TwoDiagrams}
For every $\Tth$-category $X$, the diagrams of $\V$-functors
\begin{align*}
 \xymatrix{X^\op\ar[dr]_{\yoneda_X^\op}\ar[r]^{e_X^\op} & (TX)^\op\ar[d]^{\coyoneda_{(TX)^\op}}\\ & V((TX)^\op)}
&&\text{and}&&
 \xymatrix{X^\op\ar[dr]_{\yoneda_X^\op}\ar[r]^{\coyoneda_{X^\op}} & V(X^\op)\ar[d]^{(\alpha^\op)^\Tast\kleisli-}\\ & V((TX)^\op)}
\end{align*}
commute, where in the latter diagram we assume that $X$ is representable with left adjoint $\alpha:TX\to X$ of $e_X:X\to TX$.
\end{lemma}
\begin{proof}
Let $X=(X,a)$ be a $\Tth$-category, and put $\hat{a}=\Txi a\cdot m_X^\circ$. Then, for every $x\in X$,
\[
 \coyoneda_{(TX)^\op}(e_X^\op(x))=\hat{a}^\circ(e_X^\op(x),-)=a(-,x)=\yoneda_X^\op(x)=a_0^\circ(x,\alpha(-))=(\alpha^\op)^\Tast\kleisli\coyoneda_{X^\op}(x).\qedhere
\]
\end{proof}

\begin{proposition}
Let $X$ be a representable $\Tth$-category. Then $X$ is totally cocomplete if and only if $X^\op$ is totally complete.
\end{proposition}
\begin{proof}
Assume first that $X$ is totally cocomplete. By Theorem \ref{thm:ComplIsCocompl}, $X^\op$ is complete, we write $\Inf_{X^\op}:V(X^\op)\to X^\op$ for a right adjoint of $\coyoneda_{X^\op}:X^\op\to V(X^\op)$ in $\Cat{\Tth}$. We need to prove that $\Inf_{X^\op}$ is a pseudo-homomorphism. Firstly, the diagram
\begin{equation}\label{eq:LimColim}
 \xymatrix{(PX)^\op=V((TX)^\op)\ar[dr]_{\Sup_X^\op}\ar[r]^-{V(\alpha^\op)} & V(X^\op)\ar[d]^{\Inf_{X^\op}}\\ & X^\op}
\end{equation}
of $\Tth$-functors commutes up to equivalence since the underlying diagram in $\Cat{\V}$ consists precisely of the right adjoints of the $\V$-functors of the second diagram in Lemma \ref{lem:TwoDiagrams}. Since $\Sup_X:PX\to X$ and $\alpha:TX\to X$ are left adjoints in $\Cat{\Tth}$, they are in particular pseudo-homomorphisms, hence $\Sup_X^\op$ and $V(\alpha^\op)$ are pseudo-homomorphisms. Since $V(\alpha^\op)$ is a split epimorphism in $\Cat{\V}$, also $\Inf_{X^\op}$ is a pseudo-homomorphisms. Conversely, assume now that $X^\op$ is totally complete. Hence $\Inf_{X^\op}:V(X^\op)\to X^\op$ is a pseudo-homomorphism. We show that $\Inf_{X^\op}^\op\cdot V(\alpha^\op)^\op$ is a left adjoint (=left inverse in this case) of $\yoneda_X:X\to PX$. In fact, for the duals of the underlying $\V$-functors one verifies:
\[
 \Inf_{X^\op}\cdot V(\alpha^\op)\cdot\yoneda_X^\op
=\Inf_{X^\op}\cdot V(\alpha^\op)\cdot\coyoneda_{(TX)^\op}\cdot e_X^\op
=\Inf_{X^\op}\cdot\coyoneda_{X^\op}\cdot\alpha^\op\cdot e_X^\op\simeq 1_{X_0}.\qedhere
\]

\end{proof}
From commutativity of \eqref{eq:LimColim} we also deduce that a pseudo-homomorphism $f:X\to Y$ between totally cocomplete $\Tth$-categories is cocontinuous if and only if $f^\op:X^\op\to Y^\op$ is continuous. Hence:

\begin{theorem}\label{thm:TotCoComplDual}
Taking duals defines an equivalence functor
\[
 (-)^\op:\Cocts{\Tth}\to\Cts{\Tth}
\]
which commutes with the canonical forgetful functors to $\SET$. Furthermore, $(-)^\op$ restricts to separated objects, hence
\[
 (\Cat{\Tth}_\sep)^\mP\simeq(\Repr{\Tth}_\sep)^\mV.
\]
\end{theorem}

\begin{remark}\label{rem:TildeVviaRepr}
We can write the canonical functor $\Cocts{\Tth}_\sep\to\SET^\mT$ as the composite
\[
 \Cocts{\Tth}_\sep\simeq(\Repr{\Tth}_\sep)^\mV\to\Repr{\Tth}_\sep\to\SET^\mT,
\]
hence its left adjoint sends a $\mT$-algebra $X$ to the totally cocomplete $\Tth$-category $\V^X$. Furthermore, we conclude that the monad $\mtV=\tvmonad$ on $\SET^\mT$ (see Remark \ref{rem:MonTildeV}) is also induced by the adjunction
\[
 \Cocts{\Tth}_\sep\adjunct{}{}\SET^\mT,
\]
and that $(\Repr{\Tth}_\sep)^\mV\simeq(\SET^\mT)^{\mtV}$.
\end{remark}

\begin{examples}\label{ex:TotCoComplDual}
For ordered sets, Theorem \ref{thm:TotCoComplDual} just states the trivial fact that the category $\SUP$ of $\sup$-lattices is equivalent to the category $\INF$ of $\inf$-lattices. We find it interesting to see that the topological counterpart of this result states that the category $\CONTLAT$ of continuous lattices and Scott-continuous and $\inf$-preserving maps is equivalent to the category of Eilenberg--Moore algebras for the lower Vietoris monad on the category $\STCOMP$ of stably compact spaces and spectral maps. Furthermore, by Remark \ref{rem:TildeVviaRepr}, $\CONTLAT$ is also equivalent to the category of Eilenberg--Moore algebras for the (classical) Vietoris monad on the category of compact Hausdorff spaces and continuous maps. We note that the latter equivalence was shown in \citep{Wyl81}.

For $\Tth=\Uth_{\Pp}$, the Eilenberg--Moore category $\COMPHAUS^{\mtV}$ of the monad $\mtV=\tvmonad$ on $\COMPHAUS$ (see Example \ref{ex:tildeVApp}) is equivalent to $\SET^\mP$, and this category is described in \citep{GH12} as the category of continuous lattices equipped with an internal action of $[0,\infty]$ and action-preserving morphisms of continuous lattices. A slighly different monad on $\COMPHAUS$ one obtains for $\Tth=\Uth_{\Pm}$, and the category of Eilenberg--Moore algebras of this monad is equivalent to the category of separated injective objects and left adjoint morphisms in $\UAP$.
\end{examples}

\begin{remark}
Following \citep{RW04}, we consider, for a monad $\monadfont{D}$ on a category $\catfont{C}$ where idempotents split, the full subcategory $\Spl(\catfont{C}^\monadfont{D})$ of $\catfont{C}^\monadfont{D}$ defined by the \emph{split structures}, that is, by those $\monadfont{D}$-algebras $(X,\alpha:DX\to X)$ for which exists a homomorphism $t:X\to DX$ with $\alpha\cdot t=1_X$. We put
\begin{align*}
\Cocts{\Tth}_\spl:=\Spl((\Cat{\Tth})^\mP) &&\text{and}&& \Cts{\Tth}_\spl:=\Spl((\Repr{\Tth})^\mV).
\end{align*}
Since our monads are of Kock-Z\"oberlein type, these splittings are actually adjoint to the algebra structure. Hence, a totally cocomplete separated $\Tth$-category $X$ belongs to $\Cocts{\Tth}_\spl$ if and only if $\Sup_X:PX\to X$ has a left adjoint in $\Cat{\Tth}$ (and hence in $\Cocts{\Tth}_\sep$), and a totally complete separated $\Tth$-category $X$ belongs to $\Cts{\Tth}_\spl$ if and only if $\Inf_X:VX\to X$ has a right adjoint in $\Repr{\Tth}$ (and hence in $\Cts{\Tth}_\sep$). For $X$ in $\Cocts{\Tth}_\spl$, the splitting $t:X\to PX$ of $\Sup_X:PX\to X$ is left adjoint and therefore a pseudo-homomorphism; hence, with the help of \eqref{eq:LimColim}, we see that $V(\alpha^\op)\cdot t^\op$ is a splitting of $\Inf_{X^\op}$ in $(\Repr{\Tth})^\mV$. Therefore the equivalence functor $(-)^\op:\Cocts{\Tth}\to\Cts{\Tth}$ of Theorem \ref{thm:TotCoComplDual} restricts to a functor
\[
(-)^\op:\Cocts{\Tth}_\spl\to\Cts{\Tth}_\spl;
\]
however, in general we do not obtain an equivalence as the following example shows.
\end{remark}

\begin{example}\label{ex:VXspectral} 
We consider the case of topological spaces, that is $\Tth=\Uth_\two$. For a topological space $X$, $PX$ is the filter space of $X$ (see \citep[Example 4.10]{HT10}) which is known to be spectral, and so is every split algebra for $\mP$. Since the dual of a spectral space is spectral, the image of $(-)^\op:\Cocts{\Tth}\to\Cts{\Tth}$ contains only spectral spaces. For a stably compact space $X$, $VX$ is spectral if and only if $X$ is spectral. In fact, since $\coyoneda_X:X\to VX$ is in $\STCOMP$ and a topological embedding, $X$ is spectral if $VX$ is so. If $X$ is spectral, then the topology of $VX$ is generated by the sets $V^\Diamond$ where $V$ runs through all compact opens of $X$; and for such $V$ one easily sees that $V^\Diamond$ is compact in $VX$ (using Alexander's Subbase Lemma and $(\bigcup_i V_i)^\Diamond=\bigcup_i V_i^\Diamond$). Since $VX$ is always a split algebra for $\mV$, we conclude that $(-)^\op:\Cocts{\Tth}\to\Cts{\Tth}$ is not essentially surjective on objects.

Similarly, for a compact Hausdorff space $X$, $X$ is a Stone space if and only if $\widetilde{V}X$ is a Stone space (if $X$ is Stone, then $VX$ is spectral and hence $\widetilde{V}X$ is Stone).
\end{example}

\section{The Kleisli category of the Vietoris monad}\label{sect:Kleisli}
 
Every $\Tth$-functor $r:X\to VY$ gives rise to a $\V$-matrix $\umate{r}:X\relto Y$ where $\umate{r}(x,y)=r(x)(y)$. In the sequel we are interested in the case where $X=(X,a)$ and $Y=(Y,b)$ are representable $\Tth$-categories and $r:X\to VY$ is a pseudo-homomorphism.
\begin{proposition}\label{prop:reprDist}
Let $X=(X,a)$ and $Y=(Y,b)$ be representable $\Tth$-categories. Then $r\mapsto\umate{r}$ defines a bijection between $\Repr{\Tth}(X,VY)$ and the subset of $\Mod{\V}(X_0,Y_0)$ consisting of all those $\V$-distributors $\psi:X_0\modto Y_0$ making the diagram
\begin{equation}\label{diag:QMod}
 \xymatrix{T(X_0)\ar|-{\object@{o}}[d]_a\ar|-{\object@{o}}[r]^{\Txi\psi} & T(Y_0)\ar|-{\object@{o}}[d]^b\\
 X_0\ar|-{\object@{o}}[r]_\psi & Y_0}
\end{equation}
commutative.
\end{proposition}
\begin{proof}
Let $\alpha:TX\to X$ and $\beta:TY\to Y$ be pseudo-algebra structures of $X$ and $Y$ respectively. Assume first that $r:X\to VY$ is a homomorphism. Then $\umate{r}$ is a $\Tth$-functor $\umate{r}:X^\op\otimes Y\to\V$ and hence also a $\V$-functor $\umate{r}:X_0^\op\otimes Y_0\to\V$. But the latter is equivalent to $\umate{r}$ being a $\V$-distributor $\umate{r}:X_0\modto Y_0$. Furthermore, for all $\fx\in TX$ and $y\in Y$,
\[
 \mu_Y\cdot Tr(\fx)(y)
=b\cdot\Txi\ev_Y(Tr(\fx),y)
=b\cdot\Txi\umate{r}(\fx,y);
\]
hence $r\cdot\alpha=\mu_Y\cdot Tr$ if and only if $\umate{r}\cdot a=b\cdot \Txi\umate{r}$ (note that $\umate{r}\cdot\alpha=\umate{r}\cdot a_0\cdot\alpha=\umate{r}\cdot a$ since $\umate{r}$ is a $\V$-distributor).

Assume now that $\psi:X\modto Y$ is a $\V$-distributor making the diagram \eqref{diag:QMod} commutative. Then, for every $x\in X$,
\[
 b\cdot\Txi(\psi\cdot x)\cdot e_1
=\psi\cdot a\cdot e_X\cdot x=\psi\cdot x,
\]
hence $\psi\cdot x$ can be seen a $\Tth$-distributor of type $G\kmodto Y$. We conclude that $\psi=\umate{r}$ for $r:X\to VY,\,x\mapsto\psi\cdot x$. Finally, $r$ is a $\V$-functor since $\psi$ is a $\V$-distributor, and $r$ is a homomorphism by the considerations at the end of the first part of the proof.
\end{proof}

\begin{remark}
For a $\V$-distributor $\psi:X_0\modto Y_0$, commutativity of \eqref{diag:QMod} is equivalent to $\psi\cdot\alpha=b\cdot\Txi\psi$ since $\psi\cdot a_0=\psi$. Hence, if $Y$ is of the form $Y=(Y,\beta:TY\to Y)$, \eqref{diag:QMod} commutes if and only if $\psi\cdot\alpha=\beta\cdot\Txi\psi$. Furthermore, for every pseudo-homomorphism $f:X\to Y$, the $\V$-distributor $f_*:X_0\modto Y_0$ is a morphism $f_*:X\modto Y$ in $\ReprDist{\Tth}$ since
\[
 f_*\cdot \alpha_*=(f\cdot\alpha)_*=(\beta\cdot Tf)_*=b\cdot T(f_*).
\]
More generally, given a $\V$-functor $f:X_0\to Y_0$, the $\V$-distributor $f_*:X_0\modto Y_0$ makes \eqref{diag:QMod} commutative if and only $f:X\to Y$ is a pseudo-homomorphism.
\end{remark}

We write $\ReprDist{\Tth}$ for the category with objects all representable $\Tth$-categories, and a morphisms $\psi:X\modto Y$ in $\ReprDist{\Tth}$ is a $\V$-distributor $\psi:X_0\modto Y_0$ making \eqref{diag:QMod} commutative. Composition in $\ReprDist{\Tth}$ is given by $\V$-relational composition, and $a_0:X\modto X$ is the identify arrow on $X$. Hence, $(X,a)\mapsto(X,a_0)$ defines a faithful functor
\[
 \ReprDist{\Tth}\to\Mod{\V}.
\]
The following lemma is obvious.

\begin{lemma}\label{CancelationMonoReprDist}
Let $X=(X,a)$, $Y=(Y,b)$ and $Z=(Z,c)$ be in $\ReprDist{\Tth}$ and let $\varphi:(X,a_0)\to(Y,b_0)$ and $\psi:(Y,b_0)\to(Z,c_0)$ be $\V$-distributors. If $\psi\cdot\varphi:X\modto Z$ and $\psi:Y\modto Z$ are actually morphisms in $\ReprDist{\Tth}$ and $\psi$ is mono in $\Mod{\V}$, then $\varphi:X\modto Y$ is in $\ReprDist{\Tth}$.
\end{lemma}

By definition, $f:X\to Y$ in $\Repr{\Tth}$ is downwards open (see Definition \ref{def:downwards_open}) precisely if $a\cdot \Txi(f^*)=f^*\cdot b$, and therefore:
\begin{proposition}
The following assertions are equivalent, for $f:X\to Y$ in $\Repr{\Tth}$.
\begin{eqcond}
\item $f$ is downwards open.
\item The $\V$-distributor $f^*:Y_0\modto X_0$ makes \eqref{diag:QMod} commutative (that is, $f^*:Y\modto X$ is a morphism in $\ReprDist{\Tth}$).
\item $f_*:X\modto Y$ is left adjoint in $\ReprDist{\Tth}$.
\end{eqcond}
\end{proposition}

\begin{theorem}
The Kleisli category $\Repr{\Tth}_\mV$ of $\mV$ is equivalent to $\ReprDist{\Tth}$.
\end{theorem}
\begin{proof}
It is left to show that $r\mapsto\umate{r}$ preserves composition. To see this, let $r:X\to VY$ and $s:Y\to VZ$ be in $\Repr{\Tth}_\mV$. First note that, for every $y\in Y$ and $z\in Z$,
\[
 \coyoneda_Z^\circ\cdot s_*(y,z)
=s_*(y,\coyoneda_Z(z))=[\coyoneda_Z(z),s(y)]
=s(y)(z)=\umate{s}(y,z),
\]
and therefore $\coyonmult_Y\cdot Vs:VY\to VZ$ sends $\varphi:G\kmodto Y$ to $\umate{s}\cdot\varphi$ (see Remark \ref{rem:fVastfast}). Consequently, 
\[
 \umate{s}\cdot\umate{r}(x,z)
=\umate{s}\cdot\umate{r}\cdot x(z)
=\coyonmult_Y\cdot Vs\cdot r(x)(z),
\]
for all $x\in X$ and $z\in Z$.
\end{proof}

\begin{corollary}
The functor $(-)_*:\Repr{\Tth}\to\ReprDist{\Tth}$ is left adjoint to
\begin{align*}
\ReprDist{\Tth}\to\Repr{\Tth},\;(\psi:X\modto Y)\mapsto (\psi\cdot-:VX\to VY).
\end{align*}
Here we think of an element $\varphi\in VX$ as a morphism $\varphi:G\modto X$ in $\ReprDist{\Tth}$. The units and counits are given by $\coyoneda_X:X\to VX$ and $\coyoneda_X^*:VX\modto X$ respectively.
\end{corollary}

\begin{remark}
Certainly, the adjunction above can be restricted to separated $\Tth$-categories to yield
\[
 \ReprDist{\Tth}_\sep\adjunct{(-)_*}{}\Repr{\Tth}_\sep.
\]
Furthermore, the monad $\mtV=\tvmonad$ on $\SET^\mT$ of Remark \ref{rem:MonTildeV} is also induced by the composite adjunction
\[
 \ReprDist{\Tth}_\sep\adjunct{}{}\Repr{\Tth}_\sep\adjunct{}{}\SET^\mT.
\]
The fully faithful comparison functor $(\SET^\mT)_{\mtV}\to\ReprDist{\Tth}_\sep$ induces an equivalence between the Kleisli category $(\SET^\mT)_{\mtV}$ of $\mtV$ and the full subcategory $\ReprDist{\Tth}_=$ of $\ReprDist{\Tth}$ defined by objects of the form $X=(X,\alpha:TX\to X)$ (i.e.\ where $X_0$ is a discrete $\V$-category).
\end{remark}

\begin{example}
An \emph{Esakia space} \citep{Esa74} is a Priestley space $(X,\le,\alpha)$ where the down-closure of every open (with respect to $\alpha$) subset $A\subseteq X$ is again open (with respect to $\alpha$ or, equivalently, with respect to $a=\le\cdot\alpha$). We find it worthwhile to mention that this condition just states that $i:(X,\alpha)\to (X,\le\cdot\alpha),\,x\mapsto x$ is downwards open. A \emph{morphism of Esakia spaces} (also called bounded morphism or p-morphism) is a homomorphism $f:(X,\le,\alpha)\to(Y,\le,\beta)$ such that, for all $x\in X$ and $y\in Y$ with $f(x)\le y$, there is some $x'\in X$ with $x\le x'$ and $f(x')=y$; and this condition just means that the diagram
\[
 \xymatrix{(X,\le\cdot\alpha)\ar|-{\object@{o}}[d]_{f_*}\ar|-{\object@{o}}[r]^{i_X^*} & (X,\alpha)\ar|-{\object@{o}}[d]^{f}\\
	    (Y,\le\cdot\beta)\ar|-{\object@{o}}[r]_{i_Y^*} & (Y,\beta)}
\]
commutes.
\end{example}

Motivated by the example above, we introduce the following notion.

\begin{definition}
A separated representable $\Tth$-category $(X,a)$ (with algebra structure $\alpha:TX\to X$) is an \emph{Esakia $\Tth$-category} whenever $i:(X,\alpha)\to(X,a)$ is downwards open. 
\end{definition}

\begin{proposition}\label{prop:EsakiaSplIdemp}
The following assertions are equivalent, for a separated representable $\Tth$-category $X=(X,a)$.
\begin{eqcond}
\item $X$ is an Esakia $\Tth$-category.
\item The $\V$-relation $a_0:X\relto X$ is a morphism $a_0:(X,a)\modto(X,\alpha)$ in $\ReprDist{\Tth}$, that is, the diagram
\begin{equation}\label{diag:Esa}
 \xymatrix{TX\ar[d]_\alpha\ar|-{\object@{|}}[r]^{\Txi a_0} & TX\ar[d]^\alpha\\
 X\ar|-{\object@{|}}[r]_{a_0} & X}
\end{equation}
commutes in $\Mod{\V}$.
\item $X$ is a split subobject in $\ReprDist{\Tth}$ of a $\mT$-algebra $(Y,\beta)\in\SET^\mT$.
\end{eqcond}
\end{proposition}
\begin{proof}
The equivalence (i)$\RLw$(ii) is clear by definition since $i^*=a_0$, and (ii)$\Rw$(iii) follows from $i_*\cdot i^*=a_0$. To see (iii)$\Rw$(ii), let $(Y,\beta)$ be a $\mT$-algebra and $\psi:(X,a)\modto(Y,\beta)$, $\varphi:(Y,\beta)\modto(X,a)$ be in $\ReprDist{\Tth}$ with $\varphi\cdot\psi=a_0$. Then
\[
 \psi\cdot\varphi=(\psi\cdot i_*)\cdot(i^*\cdot\varphi)
\]
in $\Mod{\V}$, $\psi\cdot\varphi$, $\psi\cdot i_*$ are in $\ReprDist{\Tth}$ and $\psi\cdot i_*$ is mono in $\Mod{\V}$, hence, by Lemma \ref{CancelationMonoReprDist}, $i^*\cdot\varphi$ is in $\ReprDist{\Tth}$. Consequently, $i^*=i^*\cdot\varphi\cdot\psi$ is in $\ReprDist{\Tth}$.
\end{proof}

\begin{example}\label{ex:HeytSplCpl}
We return to the case of topological spaces. By Example \ref{ex:VXspectral}, the monad $\mV$ on $\STCOMP$ restricts to $\SPEC$, and the Kleisli category $\SPEC_\mV$ corresponds to the full subcategory $\SPECDIST$ of $\ReprDist{\Uth_\two}$ defined by all spectral spaces. Similarly, the monad $\mtV$ on $\COMPHAUS$ restricts to $\STONE$ and the Kleisli category $\STONE_{\mtV}$ corresponds to the full subcategory $\STONEDIST$ of $\SPECDIST$ defined by all Stone spaces. Then $\STONEDIST$ is dually equivalent to the category $\BOOL_{\bot,\vee}$ of Boolean algebras and finite suprema preserving maps (see \citep{Hal62,SV88}), and $\SPECDIST$ is is dually equivalent to the category $\DLAT_{\bot,\vee}$ of distributive lattices and finite suprema preserving maps (see \citep{CLP91}). Furthermore, these dualities are closely related to duality results for Boolean algebras with operator (see \citep{KKV04}) and distributive lattices with operator (see \citep{Pet96,BKR07}). One easily verifies that $\DLAT_{\bot,\vee}$ is idempotent split complete (since it is a full subcategory of the algebraic category of sup-semilattices and homomorphisms and it is closed there under split quotients), and therefore also $\SPECDIST$ is so. Consequently, by Proposition \ref{prop:EsakiaSplIdemp}, the full subcategory of $\SPECDIST$ defined by all Esakia spaces is the idempotent split completion of $\STONEDIST$; which, by Esakia duality \citep{Esa74}, then implies that the category $\HEYT_{\bot,\vee}$ of Heyting algebras and finite suprema preserving maps is the idempotent split completion of $\BOOL_{\bot,\vee}$. 
\end{example}

A \emph{morphism $f:X\to Y$ of Esakia $\Tth$-categories} $X=(X,a)$ and $Y=(Y,b)$ (where $b=b_0\cdot\beta$) is a homomorphism making
\[
 \xymatrix{(X,a)\ar|-{\object@{o}}[d]_{f_*}\ar|-{\object@{o}}[r]^{i_X^*} & (X,\alpha)\ar|-{\object@{o}}[d]^{f}\\
	    (Y,b)\ar|-{\object@{o}}[r]_{i_Y^*} & (Y,\beta)}
\]
commutative in $\ReprDist{\Tth}$, which can be equivalently expressed by saying that either of the diagrams
\begin{align*}
 \xymatrix{X\ar[r]^-{\mate{a_0}}\ar[d]_f & V(X,\alpha)\ar[d]^{Vf}\\ Y\ar[r]_-{\mate{b_0}} & V(Y,\beta)}
&&\text{or}&&
 \xymatrix{(X,\alpha)\ar[r]^-{\mate{a_0}}\ar[d]_f & \widetilde{V}(X,\alpha)\ar[d]^{\widetilde{V}f}\\ 
           (Y,\beta)\ar[r]_-{\mate{b_0}} & \widetilde{V}(Y,\beta)}
\end{align*}
commutes in $\Repr{\Tth}$ or $\SET^\mT$ respectively. Let now $X=(X,\alpha)$ be in $\SET^\mT$ and $r:X\to\widetilde{V}X$ be a homomorphism. Then, with $a_0:=\umate{r}:X\relto X$, the diagram \eqref{diag:Esa} commutes by Proposition \ref{prop:reprDist}. Therefore $(X,a_0\cdot\alpha)$ is an Esakia $\Tth$-category provided that $a_0$ is a separated $\V$-category structure on the set $X$. Summing up, we have identified the category of Esakia $\Tth$-categories and morphisms as the full subcategory of the category $\Coalg(\widetilde{V})$ of coalgebras for $\widetilde{V}:\SET^\mT\to\SET^\mT$ defined by those coalgebras $r:X\to\widetilde{V}X$ whose mate $a_0:=\umate{r}:X\relto X$ is a separated $\V$-category structure on the set $X$. This observation represents a generalisation of \citep{DG03}.

\def\cprime{$'$}


\begin{thebibliography}{61}\setlength{\itemsep}{1ex}\small
\providecommand{\natexlab}[1]{#1}
\providecommand{\url}[1]{\texttt{#1}}
\providecommand{\urlprefix}{URL }
\providecommand{\eprint}[2][]{\url{#2}}

\bibitem[{Banaschewski \emph{et~al.}(2006)Banaschewski, Lowen and {Van
  Olmen}}]{BLO06}
\textsc{Banaschewski, B.}, \textsc{Lowen, R.} and \textsc{{Van Olmen}, C.}
  (2006), Sober approach spaces, \emph{Topology Appl.} \textbf{153}~(16),
  3059--3070.

\bibitem[{Barr(1970)}]{Bar70}
\textsc{Barr, M.} (1970), Relational algebras, in \emph{Reports of the Midwest
  Category Seminar, IV}, pages 39--55, Lecture Notes in Mathematics, Vol. 137.
  Springer, Berlin.

\bibitem[{Betti \emph{et~al.}(1983)Betti, Carboni, Street and Walters}]{BCSW83}
\textsc{Betti, R.}, \textsc{Carboni, A.}, \textsc{Street, R.} and
  \textsc{Walters, R.} (1983), Variation through enrichment, \emph{J. Pure
  Appl. Algebra} \textbf{29}~(2), 109--127.

\bibitem[{Bonsangue \emph{et~al.}(2007)Bonsangue, Kurz and Rewitzky}]{BKR07}
\textsc{Bonsangue, M.~M.}, \textsc{Kurz, A.} and \textsc{Rewitzky, I.~M.}
  (2007), Coalgebraic representations of distributive lattices with operators,
  \emph{Topology Appl.} \textbf{154}~(4), 778--791.

\bibitem[{Cignoli \emph{et~al.}(1991)Cignoli, Lafalce and Petrovich}]{CLP91}
\textsc{Cignoli, R.}, \textsc{Lafalce, S.} and \textsc{Petrovich, A.} (1991),
  Remarks on {P}riestley duality for distributive lattices, \emph{Order}
  \textbf{8}~(3), 299--315.

\bibitem[{Clementino and Hofmann(2003)}]{CH03}
\textsc{Clementino, M.~M.} and \textsc{Hofmann, D.} (2003), Topological
  features of lax algebras, \emph{Appl. Categ. Structures} \textbf{11}~(3),
  267--286.

\bibitem[{Clementino and Hofmann(2004)}]{CH04a}
\textsc{Clementino, M.~M.} and \textsc{Hofmann, D.} (2004), Effective descent
  morphisms in categories of lax algebras, \emph{Appl. Categ. Structures}
  \textbf{12}~(5-6), 413--425.

\bibitem[{Clementino and Hofmann(2009{\natexlab{a}})}]{CH09}
\textsc{Clementino, M.~M.} and \textsc{Hofmann, D.} (2009{\natexlab{a}}),
  Lawvere completeness in topology, \emph{Appl. Categ. Structures}
  \textbf{17}~(2), 175--210, \eprint{arXiv:math.CT/0704.3976}.

\bibitem[{Clementino and Hofmann(2009{\natexlab{b}})}]{CH09a}
\textsc{Clementino, M.~M.} and \textsc{Hofmann, D.} (2009{\natexlab{b}}),
  Relative injectivity as cocompleteness for a class of distributors,
  \emph{Theory Appl. Categ.} \textbf{21}, No. 12, 210--230,
  \eprint{arXiv:math.CT/0807.4123}.

\bibitem[{Clementino \emph{et~al.}(2003)Clementino, Hofmann and Tholen}]{CHT03}
\textsc{Clementino, M.~M.}, \textsc{Hofmann, D.} and \textsc{Tholen, W.}
  (2003), Exponentiability in categories of lax algebras, \emph{Theory Appl.
  Categ.} \textbf{11}, No. 15, 337--352.

\bibitem[{Clementino and Tholen(1997)}]{CT97}
\textsc{Clementino, M.~M.} and \textsc{Tholen, W.} (1997), A characterization
  of the {V}ietoris topology, in \emph{Proceedings of the 12th {S}ummer
  {C}onference on {G}eneral {T}opology and its {A}pplications ({N}orth {B}ay,
  {ON}, 1997)}, volume~22, pages 71--95.

\bibitem[{Clementino and Tholen(2003)}]{CT03}
\textsc{Clementino, M.~M.} and \textsc{Tholen, W.} (2003), Metric, topology and
  multicategory---a common approach, \emph{J. Pure Appl. Algebra}
  \textbf{179}~(1-2), 13--47.

\bibitem[{Davey and Galati(2003)}]{DG03}
\textsc{Davey, B.~A.} and \textsc{Galati, J.~C.} (2003), A coalgebraic view of
  {H}eyting duality, \emph{Studia Logica} \textbf{75}~(3), 259--270.

\bibitem[{Day and Kelly(1970)}]{DK70}
\textsc{Day, B.~J.} and \textsc{Kelly, G.~M.} (1970), On topologically quotient
  maps preserved by pullbacks or products, \emph{Proc. Cambridge Philos. Soc.}
  \textbf{67}, 553--558.

\bibitem[{Eilenberg and Kelly(1966)}]{EK66}
\textsc{Eilenberg, S.} and \textsc{Kelly, G.~M.} (1966), Closed categories, in
  \emph{Proc. Conf. Categorical Algebra (La Jolla, Calif., 1965)}, pages
  421--562, Springer, New York.

\bibitem[{Esakia(1974)}]{Esa74}
\textsc{Esakia, L.} (1974), Topological {K}ripke models, \emph{Dokl. Akad. Nauk
  SSSR} \textbf{214}, 298--301.

\bibitem[{Escard{\'o}(1998)}]{Esc98}
\textsc{Escard{\'o}, M.~H.} (1998), Properly injective spaces and function
  spaces, \emph{Topology Appl.} \textbf{89}~(1-2), 75--120.

\bibitem[{Gierz \emph{et~al.}(1980)Gierz, Hofmann, Keimel, Lawson, Mislove and
  Scott}]{GHK+80}
\textsc{Gierz, G.}, \textsc{Hofmann, K.~H.}, \textsc{Keimel, K.},
  \textsc{Lawson, J.~D.}, \textsc{Mislove, M.~W.} and \textsc{Scott, D.~S.}
  (1980), \emph{A compendium of continuous lattices}, Springer-Verlag, Berlin,
  xx+371 pages.

\bibitem[{Gutierres and Hofmann(2012)}]{GH12}
\textsc{Gutierres, G.} and \textsc{Hofmann, D.} (2012), Approaching metric
  domains, \emph{accepted for publication in Appl. Categ. Structures},
  \eprint{arXiv:math.GN/1103.4744}.

\bibitem[{Halmos(1962)}]{Hal62}
\textsc{Halmos, P.~R.} (1962), \emph{Algebraic logic}, Chelsea Publishing Co.,
  New York, 271 pages.

\bibitem[{Hermida(2000)}]{Her00}
\textsc{Hermida, C.} (2000), Representable multicategories, \emph{Adv. Math.}
  \textbf{151}~(2), 164--225.

\bibitem[{Hochster(1969)}]{Hoc69}
\textsc{Hochster, M.} (1969), Prime ideal structure in commutative rings,
  \emph{Trans. Amer. Math. Soc.} \textbf{142}, 43--60.

\bibitem[{Hofmann(2007)}]{Hof07}
\textsc{Hofmann, D.} (2007), Topological theories and closed objects,
  \emph{Adv. Math.} \textbf{215}~(2), 789--824.

\bibitem[{Hofmann(2011)}]{Hof11}
\textsc{Hofmann, D.} (2011), Injective spaces via adjunction, \emph{J. Pure
  Appl. Algebra} \textbf{215}~(3), 283--302, \eprint{arXiv:math.CT/0804.0326}.

\bibitem[{Hofmann(2013)}]{Hof13}
\textsc{Hofmann, D.} (2013), Duality for distributive spaces, \emph{Theory
  Appl. Categ.} \textbf{28}~(3), 66--122, \eprint{arXiv:math.CT/1009.3892}.

\bibitem[{Hofmann and Stubbe(2011)}]{HS11}
\textsc{Hofmann, D.} and \textsc{Stubbe, I.} (2011), Towards {S}tone duality
  for topological theories, \emph{Topology Appl.,} \textbf{158}~(7), 913--925,
  \eprint{arXiv:math.CT/1004.0160}.

\bibitem[{Hofmann and Tholen(2010)}]{HT10}
\textsc{Hofmann, D.} and \textsc{Tholen, W.} (2010), {L}awvere completion and
  separation via closure, \emph{Appl. Categ. Structures} \textbf{18}~(3),
  259--287, \eprint{arXiv:math.CT/0801.0199}.

\bibitem[{Hofmann and Waszkiewicz(2011)}]{HW11}
\textsc{Hofmann, D.} and \textsc{Waszkiewicz, P.} (2011), Approximation in
  quantale-enriched categories, \emph{Topology Appl.} \textbf{158}~(8),
  963--977, \eprint{arXiv:math.CT/1004.2228}.

\bibitem[{Isbell(1975)}]{Isb75}
\textsc{Isbell, J.~R.} (1975), Meet-continuous lattices, in \emph{Symposia
  {M}athematica, {V}ol. {XVI} ({C}onvegno sui {G}ruppi {T}opologici e {G}ruppi
  di {L}ie, {INDAM}, {R}ome, 1974)}, pages 41--54, Academic Press, London.

\bibitem[{Isbell(1986)}]{Isb86}
\textsc{Isbell, J.~R.} (1986), General function spaces, products and continuous
  lattices, \emph{Math. Proc. Cambridge Philos. Soc.} \textbf{100}~(2),
  193--205.

\bibitem[{Johnstone(1986)}]{Joh86}
\textsc{Johnstone, P.~T.} (1986), \emph{Stone spaces}, volume~3 of
  \emph{Cambridge Studies in Advanced Mathematics}, Cambridge University Press,
  Cambridge, xxii+370 pages, reprint of the 1982 edition.

\bibitem[{Jung(2004)}]{Jun04}
\textsc{Jung, A.} (2004), Stably compact spaces and the probabilistic
  powerspace construction, in J.~Desharnais and P.~Panangaden, editors,
  \emph{Domain-theoretic Methods in Probabilistic Processes}, volume~87, 15pp.

\bibitem[{Kelly(1982)}]{Kel82}
\textsc{Kelly, G.~M.} (1982), \emph{Basic concepts of enriched category
  theory}, volume~64 of \emph{London Mathematical Society Lecture Note Series},
  Cambridge University Press, Cambridge, 245 pages, also in: Repr. Theory Appl.
  Categ. No.~10 (2005), 1--136.

\bibitem[{Kupke \emph{et~al.}(2004)Kupke, Kurz and Venema}]{KKV04}
\textsc{Kupke, C.}, \textsc{Kurz, A.} and \textsc{Venema, Y.} (2004), Stone
  coalgebras, \emph{Theoret. Comput. Sci.} \textbf{327}~(1-2), 109--134.

\bibitem[{Lawson(2011)}]{Law11}
\textsc{Lawson, J.} (2011), Stably compact spaces, \emph{Math. Structures
  Comput. Sci.} \textbf{21}~(1), 125--169.

\bibitem[{Lawvere(1973)}]{Law73}
\textsc{Lawvere, F.~W.} (1973), Metric spaces, generalized logic, and closed
  categories, \emph{Rendiconti del Seminario Matematico e Fisico di Milano}
  \textbf{43}, 135--166, {Republished in: Reprints in Theory and Applications
  of Categories, No. 1 (2002) pp 1-37}.

\bibitem[{Linton(1969)}]{Lin69}
\textsc{Linton, F. E.~J.} (1969), Coequalizers in categories of algebras, in
  \emph{Sem. on {T}riples and {C}ategorical {H}omology {T}heory ({ETH},
  {Z}\"urich, 1966/67)}, pages 75--90, Springer, Berlin.

\bibitem[{Lowen(1989)}]{Low89}
\textsc{Lowen, R.} (1989), Approach spaces: a common supercategory of {TOP} and
  {MET}, \emph{Math. Nachr.} \textbf{141}, 183--226.

\bibitem[{Lowen(1997)}]{Low97}
\textsc{Lowen, R.} (1997), \emph{Approach spaces}, Oxford Mathematical
  Monographs, The Clarendon Press Oxford University Press, New York, x+253
  pages, the missing link in the topology-uniformity-metric triad, Oxford
  Science Publications.

\bibitem[{Lucyshyn-Wright(2011)}]{Luc11}
\textsc{Lucyshyn-Wright, R. B.~B.} (2011), Domains occur among spaces as strict
  algebras among lax, \emph{Math. Structures Comput. Sci.} \textbf{21}~(3),
  647--670.

\bibitem[{Manes(1969)}]{Man69}
\textsc{Manes, E.~G.} (1969), {A triple theoretic construction of compact
  algebras.}, {Sem. on Triples and Categorical Homology Theory, ETH Z{\"u}rich
  1966/67, Lect. Notes Math. 80, 91-118 (1969).}

\bibitem[{Manes(1976)}]{Man76}
\textsc{Manes, E.~G.} (1976), \emph{Algebraic theories}, Springer-Verlag, New
  York, vii+356 pages, graduate Texts in Mathematics, No. 26.

\bibitem[{Nachbin(1950)}]{Nac50}
\textsc{Nachbin, L.} (1950), \emph{Topologia e Ordem}, Univ. of Chicago Press,
  in {P}ortuguese, {E}nglish translation: Topology and Order, Van Nostrand,
  Princeton (1965).

\bibitem[{Petrovich(1996)}]{Pet96}
\textsc{Petrovich, A.} (1996), Distributive lattices with an operator,
  \emph{Studia Logica} \textbf{56}~(1-2), 205--224, special issue on Priestley
  duality.

\bibitem[{Pisani(1999)}]{Pis99}
\textsc{Pisani, C.} (1999), Convergence in exponentiable spaces, \emph{Theory
  Appl. Categ.} \textbf{5}, No. 6, 148--162.

\bibitem[{Priestley(1970)}]{Pri70}
\textsc{Priestley, H.~A.} (1970), Representation of distributive lattices by
  means of ordered stone spaces, \emph{Bull. London Math. Soc.} \textbf{2},
  186--190.

\bibitem[{Rosebrugh and Wood(1994)}]{RW94}
\textsc{Rosebrugh, R.} and \textsc{Wood, R.~J.} (1994), Constructive complete
  distributivity. {IV}, \emph{Appl. Categ. Structures} \textbf{2}~(2),
  119--144.

\bibitem[{Rosebrugh and Wood(2004)}]{RW04}
\textsc{Rosebrugh, R.} and \textsc{Wood, R.~J.} (2004), Split structures,
  \emph{Theory Appl. Categ.} \textbf{13}, No. 12, 172--183.

\bibitem[{Sambin and Vaccaro(1988)}]{SV88}
\textsc{Sambin, G.} and \textsc{Vaccaro, V.} (1988), Topology and duality in
  modal logic, \emph{Ann. Pure Appl. Logic} \textbf{37}~(3), 249--296.

\bibitem[{Schubert and Seal(2008)}]{SS08}
\textsc{Schubert, C.} and \textsc{Seal, G.~J.} (2008), Extensions in the theory
  of lax algebras, \emph{Theory Appl. Categ.} \textbf{21}, No. 7, 118--151.

\bibitem[{Schwarz(1984)}]{Sch84}
\textsc{Schwarz, F.} (1984), Product compatible reflectors and
  exponentiability, in \emph{Categorical topology (Toledo, Ohio, 1983)},
  volume~5 of \emph{Sigma Ser. Pure Math.}, pages 505--522, Heldermann, Berlin.

\bibitem[{Seal(2005)}]{Sea05}
\textsc{Seal, G.~J.} (2005), Canonical and op-canonical lax algebras,
  \emph{Theory Appl. Categ.} \textbf{14}, 221--243.

\bibitem[{Seal(2009)}]{Sea09}
\textsc{Seal, G.~J.} (2009), A {K}leisli-based approach to lax algebras,
  \emph{Appl. Categ. Structures} \textbf{17}~(1), 75--89.

\bibitem[{Simmons(1982)}]{Sim82}
\textsc{Simmons, H.} (1982), A couple of triples, \emph{Topology Appl.}
  \textbf{13}~(2), 201--223.

\bibitem[{Stone(1936)}]{Sto36}
\textsc{Stone, M.~H.} (1936), The theory of representations for {B}oolean
  algebras, \emph{Trans. Amer. Math. Soc.} \textbf{40}~(1), 37--111.

\bibitem[{Stone(1938)}]{Sto38}
\textsc{Stone, M.~H.} (1938), Topological representations of distributive
  lattices and {B}rouwerian logics, \emph{\v{C}asopis pro p\v{e}stov\'an\'i
  matematiky a fysiky} \textbf{67}~(1), 1--25,
  \eprint{http://dml.cz/handle/10338.dmlcz/124080}.

\bibitem[{Tholen(2009)}]{Tho09}
\textsc{Tholen, W.} (2009), Ordered topological structures, \emph{Topology
  Appl.} \textbf{156}~(12), 2148--2157.

\bibitem[{Van~Olmen(2005)}]{Olm05}
\textsc{Van~Olmen, C.} (2005), \emph{A study of the interaction between frame
  theory and approach theory}, Ph.D. thesis, University of Antwerp.

\bibitem[{Vietoris(1922)}]{Vie22}
\textsc{Vietoris, L.} (1922), Bereiche zweiter {O}rdnung, \emph{Monatsh. Math.
  Phys.} \textbf{32}~(1), 258--280.

\bibitem[{Wood(2004)}]{Woo04}
\textsc{Wood, R.} (2004), Ordered sets via adjunction, in \emph{Categorical
  foundations: special topics in order, topology, algebra, and sheaf theory},
  volume~97 of \emph{Encyclopedia Math. Appl.}, pages 5--47, Cambridge Univ.
  Press, Cambridge.

\bibitem[{Wyler(1981)}]{Wyl81}
\textsc{Wyler, O.} (1981), Algebraic theories of continuous lattices,
  {Continuous lattices, Proc. Conf., Bremen 1979, Lect. Notes Math. 871,
  390-413 (1981)}.

\end{thebibliography}

\end{document}